\documentclass[11pt]{amsart}

\usepackage{amsmath,amssymb,epsfig}
\usepackage{graphicx}
\usepackage{fullpage}
\usepackage{amscd}
\usepackage{amsfonts, amssymb} 
\usepackage{stmaryrd}
\usepackage[all]{xy}
\usepackage{upgreek}
\usepackage{latexsym}
\usepackage{cite}
\usepackage[usenames]{color}
\usepackage{xcolor}
\usepackage{datetime2} 
\usepackage{url}
\usepackage{boxedminipage}
\usepackage{amsbsy}
\usepackage{framed}
\usepackage{hyperref}
\hypersetup{colorlinks}

\usepackage{enumitem}
\usepackage{bm}

\usepackage{longtable}
\usepackage{tikz}
\usepackage{accents}

\newcommand{\fg}{\mathfrak g}

\newcommand{\fri}{\mathfrak i}
\DeclareRobustCommand{\frL}{\mathfrak{L}}

\newcommand{\cC}{{\mathcal C}}
\newcommand{\cB}{{\mathcal B}}

\newcommand{\cL}{{\mathcal L}}

\newcommand{\cP}{{\mathcal P}}

\newcommand{\cX}{{\mathcal X}}

\newcommand{\N}{\mathbb{N}}
\newcommand{\R}{\mathbb{R}}
\newcommand{\C}{\mathbb{C}}
\newcommand{\Z}{\mathbb{Z}}

\usepackage{mathrsfs}

\newcommand{\scF}{\mathscr{F}}
\newcommand{\scA}{\mathscr{A}}
\newcommand{\scB}{\mathscr{B}}
\newcommand{\scC}{\mathscr{C}}
\newcommand{\scD}{\mathscr{D}}

\newcommand{\scH}{\mathscr{H}}
\newcommand{\scJ}{\mathscr{J}}
\newcommand{\scG}{\mathscr{G}}
\newcommand{\scL}{\mathscr{L}}
\newcommand{\scM}{\mathscr{M}}
\newcommand{\scO}{\mathscr{O}}
 \newcommand{\scN}{\mathscr{N}}
\newcommand{\scS}{\mathscr{S}}
\newcommand{\scT}{\mathscr{T}}
\newcommand{\scP}{\mathscr{P}}

\newcommand{\sfC}{\mathsf{C}}

\newcommand{\Euc}{\mathsf{Euc}}
\newcommand{\Set}{\mathsf{Set}}

\newcommand{\Aff}{\mathsf{\cin Aff}}
\newcommand{\cring}{\cin\mathsf{Ring}}
\newcommand{\CGA}{\mathsf{CGA}}
\newcommand{\Mod}{\mathsf{Mod}}

\newcommand{\DiffSp}{\mathsf{DiffSpace}}
\newcommand{\LCRS}{\mathsf{L\cin RS}}

\newcommand{\CDGA}{\mathsf{CDGA}}
\newcommand{\Open}{\mathsf{Open}}

\newcommand{\Psh}{\mathsf{Psh}}
\newcommand{\Sh}{\mathsf{Sh}}
\newcommand{\sh}{\mathbf{sh}}

\DeclareMathAlphabet{\mathpzc}{OT1}{pzc}{m}{it}

\newcommand{\inv}{^{-1}}

\newcommand{\op}[1]{{#1}^{\mbox{\sf{\tiny{op}}}}}
\newcommand{\cin}{C^\infty}

\newcommand{\bOmega}{\bm{\Upomega}}
\newcommand{\bi}{\bm{i}}

\newcommand{\bd}{\bm{d}}

\DeclareMathOperator{\supp}{supp}

\DeclareMathOperator{\Der}{Der}
\DeclareMathOperator{\cDer}{C^\infty Der}
\DeclareMathOperator{\Dert}{\mathit{Der}\,}
\DeclareMathOperator{\Hom}{Hom}

\DeclareMathOperator{\Spec}{Spec}

\DeclareMathOperator{\pr}{\mathsf{pr}}

\def \cHom {{\mathcal Hom}}
\def \Dert {{\mathcal Der\,}}

\newcommand{\bfi}{ \bm{\fri}}

\theoremstyle{definition}
\newtheorem{thm}{Theorem}[section]
\newtheorem{lemma}[thm]{Lemma}
\newtheorem{theorem}[thm]{Theorem}

\newtheorem{proposition}[thm]{Proposition}
\newtheorem{corollary}[thm]{Corollary}
\newtheorem*{corollary*}{Corollary}
\newtheorem*{claim*}{Claim}

\newtheorem{definition}[thm]{Definition}
\newtheorem{remark}[thm]{Remark}
\newtheorem*{remark*}{Remark}
\newtheorem{example}[thm]{Example}
\newtheorem{notation}[thm]{Notation}
\newtheorem {construction}[thm]{Construction}

\numberwithin{equation}{thm}

\begin{document}
\title{Cartan calculus for $C^\infty$-ringed spaces} 
\author{ Eugene Lerman}

\setcounter{tocdepth}{1}

\begin{abstract} In an earlier paper \cite{LdR} we constructed a complex of
  differential forms on a local $C^\infty$-ringed space.  In this
  paper we define a sheaf of vector fields (``the tangent sheaf'') on a
  local $C^\infty$-ringed space, define contractions of vector fields
  and forms, Lie derivatives of forms with respect to vector fields,
  and show that the standard equations of
  Cartan calculus hold for vector fields and differential forms on
  local $C^\infty$-ringed spaces.
  \end{abstract}
\maketitle
\tableofcontents

\section{Introduction}

The goal of this paper is to construct an analogue of the usual Cartan
calculus on manifolds for a larger class of spaces on which the notions
of  vector fields and of differential forms make sense.
Specifically we carry out the construction 
for local $\cin$-ringed spaces.   Local $\cin$-ringed spaces   includes manifolds,
manifolds with corners, differential spaces in the sense of Sikorski \cite{Sn},
$C^\infty$-differentiable  space of Navarro Gonz\'{a}lez and Sancho de
Salas \cite{NGSS}, affine $\cin$-schemes of Dubuc \cite{Dubuc} and
Kreck's stratifolds \cite{Kreck}.

The paper is a third in a series of several papers developing 
  differential geometry on $\cin$-ringed spaces (as opposed to {\em
  algebraic} geometry as developed by Joyce \cite{Joy} following the
pioneering work of Dubuc \cite{Dubuc}).  
The first paper \cite{KL} (joint with Yael Karshon)
proves the existence of flows of vector fields.   The second
\cite{LdR} constructs differential forms.  In the case of differential
spaces the module of these differential  forms maps surjectively onto
the module of differential forms constructed by Mostow
\cite{Mostow}.

Singular spaces, that is, spaces that are not manifolds, arise
naturally in differential geometry and its applications to physics
(e.g., \cite{SL},  \cite{Sn}),
engineering (e.g., \cite{CGKS}) and economics (e.g., \cite{Simsek}).
There are many approaches to differential geometry on singular spaces
and there is a vast literature which we will not attempt to survey.
For us a singular space is a local $\cin$-ringed space (diffeological
spaces \cite{IZ} is another popular approach; it is less well suited
for handling vector fields).

Our ultimate goal of developing differential geometry on $\cin$-ringed
spaces is to have enough tools on hand in order to be able to do
geometric mechanics on these spaces.  For example, {\em configuration}
spaces of mechanical linkages have been studied intensely from various
points of view.  Again, the literature is to vast to survey here. The
{\em phase} spaces of linkages and the dynamics on these phase spaces
are less well understood.  We expect them to be hybrid dynamical
systems with the continuous time dynamics taking place on differential
spaces.  A coordinate-free derivation of Euler-Lagrange equations on
manifolds uses Cartan calculus.  Since the coordinate-based version of
Euler-Lagrange equations on differential spaces does not make sense,
we need Cartan calculus on differential spaces just to write
down the Euler-Lagrange equations.  But it is no harder to develop
Cartan calculus on arbitrary $\cin$-ringed spaces, which is what we do
here.\\

Recall one possible formulation of Cartan calculus on manifolds.
First of all, we can think of the space $\chi(M)$ vector fields on a $\cin$-manifold $M$
as the Lie algebra of derivations of the algebra $\cin(M)$ of smooth
functions (the Lie bracket of two vector fields is their commutator:
$[X,Y] = X\circ Y - Y\circ X$, as usual).   
We also  have a graded commutative algebra of
differential forms $\bOmega_{dR}^\bullet (M)= \{\bOmega_{dR}^n
(M)\}_{n=0}^\infty$ together with the exterior derivative $d: \bOmega_{dR}^\bullet (M) \to
\bOmega_{dR}^{\bullet +1}  (M)$.  The exterior derivative is a degree +1
derivation of the exterior algebra $\bOmega_{dR}^\bullet(M)$ hence the triple
$(\bOmega_{dR}^\bullet (M), \wedge, d)$ is a commutative differential graded
algebra (a CDGA). 
For each vector field $v$ on $M$ we have a contraction
$\imath_v: \bOmega_{dR}^\bullet (M) \to \bOmega_{dR}^{\bullet -1}
(M)$.  These contractions are degree -1 derivations of the exterior
algebra.  Consequently we have a
map
\[
\imath: \chi (M)\to \Der^{-1} (\bOmega_{dR}^\bullet (M) ),
\] 
of $\cin(M)$-modules.
For every vector field $v$ the graded commutator
\[
[d, \imath_v] = d \circ \imath_v + \imath_v \circ d
\]
of a degree 1 derivation $d$ and a degree -1 derivation $\imath_v$ 
is a degree 0 derivation of the  graded commutative algebra
$\bOmega_{dR}^\bullet (M)$.  (Recall that the graded commutator of two
  derivations $X$, $Y$ of degrees $|X|$ and $|Y|$, respectively, is by
  definition
  \[
[X,Y] := X\circ Y - (-1) ^{|X||Y|} Y\circ X.
\]
The commutator $[X,Y]$  is a derivation of degree $|X| +|Y|$.)

The commutator  $[d, \imath_v] $ happens to
agree with the Lie derivative $\cL_v$, which is usually defined
geometrically, by way of the local flow of the vector field $v$.  This
fact is known as Cartan's magic formula:
\begin{equation} \label{eq:c-1}
\cL_v = d \circ \imath_v + \imath_v \circ d.
\end{equation}
It is well-known that the
exterior derivative, contractions and Lie derivatives satisfy the
following identities for all vector fields $v,w$ on $M$:
\begin{align}
  \cL_v \circ d &= d\circ \cL_v \label{eq:c0}\\
\cL_{[v,w]} &= \cL_v \circ \cL_w - \cL_w \circ \cL_v\, (=: 
  [\cL_v,
  \cL_w] )\label{eq:c1}  \\
\cL_v \circ \imath_w - \imath_w \circ \cL_v&= \imath_{[v,w]}\label{eq:c2}\\
\imath_v \circ \imath_w +\imath_w\circ \imath _v &=0. \label{eq:c3}
\end{align}
The identities \eqref{eq:c-1} - \eqref{eq:c3}
are known as Cartan calculus.\\[6pt]
Recall that, roughly speaking, a {\sf  $\cin$-ring } is a (nonempty) {\em set} $\scC$ together with
operations (i.e., functions) \label{page:c-ring}
\[
g_\scC:\scC^m\to \scC
\]
for all $m$ and all $g\in C^\infty (\R^m)$
such that for all
$n,m\geq 0$, %
 all $g\in C^\infty(\R^m)$ and all
$f_1, \ldots, f_m\in C^\infty (\R^n)$
\begin{equation} \label{eq:2.assoc}
({g\circ(f_1,\ldots, f_m)})_\scC (c_1,\ldots, c_n) =
g_\scC({(f_1)}_\scC(c_1,\ldots, c_n), \ldots, {(f_m)}_\scC(c_1,\ldots, c_n))
\end{equation}
for all $(c_1, \ldots, c_n) \in \scC^n$.  The precise definition is
Definition~\ref{def:C2} below.  If $M$ is a manifold then the algebra
$\cin(M)$ of smooth function on $M$ is a $\cin$-ring: for any $n\geq
0$ and for any $h\in \cin(\R^n)$ the corresponding operation
$h_{\cin(M)} :(\cin(M))^n \to \cin(M)$ is given by
\[
h_{\cin(M)}(f_1,\ldots, f_n) := h\circ (f_1,\ldots, f_n) 
\]  
for all $f_1,\ldots, f_n\in \cin(M)$ (on the right $(f_1,\ldots, f_n)$
is a function from $M$ to $\R^n$).  The real line $\R$ is
also a $\cin$-ring: for all $n\geq
0$ and  any $h\in \cin(\R^n)$ the corresponding operation $h_\R:
\R^n\to \R$ is given by
\[
h_\R(x_1,\ldots, x_n): = h(x_1,\ldots, x_n).
\]

A map/morphism from a $\cin$-ring $\scA$ to a $\cin$-ring $\scB$ is a
function $\varphi:\scA\to \scB$ that preserves all operations:
\[
\varphi (h_\scA(a_1,\ldots, a_n)) = h_\scB (\varphi(a_1),\ldots \varphi(a_n))
\]
for all $n$, all $a_1,\ldots, a_n \in \scA$ and al $h\in \cin(\R^n)$.
A $\cin$-ring $\scC$ is {\sf local} iff there is a unique map $p:\scC
\to \R$ of $\cin$-rings with $p(1_\scC) = 1$.   A {\sf local $\cin$-ringed space}
($\LCRS$) is a pair $(M,\scA)$ where $M$ is a topological space and
$\scA$ is a sheaf of $\cin$-rings with the additional property that
the stalks of the sheaf are local $\cin$-rings.   We will refer to the
sheaf $\scA$ as the {\sf structure sheaf} of an $\LCRS$ $(M,\scA)$.

It is not hard to formulate and prove the Cartan calculus for a
$\cin$-ring.  
One starts with the notion of
a {\sf module} $\scM$ over a $\cin$-ring $\scC$: this is just a module
over the commutative $\R$-algebra underlying $\scC$ (any $\cin$-ring
has an underlying $\R$-algebra, see Remark~\ref{rmrk:cin-R-alg}).  A {\sf
  $\cin$-ring derivation} of a $\cin$-ring $\scC$ with values in a
module $\scM$ is a function $v:\scC \to \scM$ so that for any $n$, any
$h\in \cin(\R^n)$ and any $n$-tuple $(c_1,\ldots, c_n) \in \scC^n$
\[
v (h_\scC(c_1,\ldots, c_n)) = \sum _i (\partial _i h)_\scC(c_1,\ldots,
c_n))\cdot v(c_i).
\]
Here $\cdot$ denotes the action of the $\cin$-ring $\scC$ on the
module $\scM$.  $\cin$-derivations with values in a module $\scM$ form
a $\scC$-module which we denote by $\cDer (\scC, \scM)$.  We
abbreviate $\cDer(\scC, \scC)$ as $\cDer(\scC)$.   One can show that
for any two derivations $v,w\in \cDer(\scC)$ their commutator
\[
[v,w]: = v\circ w - w\circ v
\]  
is again a $\cin$-ring derivation.   (The proof boils down to the fact
that mixed partials commute.)  Consequently $\cDer(\scC)$ is a
Lie algebra.

The module {\sf of $\cin$-K\"ahler differentials} of a
$\cin$-ring $\scC$ is a $\scC$-module $\Omega_\scC^1$ together with a
{\em universal}
derivation $d_\scC:\scC \to \Omega_\scC^1$.  As in commutative algebra
``universal'' means that for any module
$\scM$ and any derivation $v:\scC\to \scM$ there exists a unique map
$\iota(v): \Omega^1_\scC\to \scM$ with $
\iota(v) \circ d_\scC = v$.
In other words, 
\[
d_\scC^*:\Hom_\scC (\Omega^1_\scC, \scM) \to \cDer(\scC, \scM),\qquad
\varphi\to \varphi \circ d_\scC
\]  
is an isomorphism of $\scC$-modules.

For any $\scC$-module $\scM$ the exterior powers $\scM^\bullet = \{\Lambda^k
\scM\}_{k\geq 0}$ form form a commutative graded algebra (a CGA) under
the exterior product $\wedge$.
Consequently for any $\cin$-ring $\scC$ we have the CGA  $(\Lambda
^\bullet \Omega^1_\scC, \wedge)$.  One can
show \cite{LdR} that the universal derivation $d_\scC:\scC\to \Omega^1_\scC$
extends uniquely to a degree $+1$ derivation $d$ 
 of $(\Lambda^\bullet
\Omega^1_\scC, \wedge)$ with $d\circ d =0$.  This is in complete
analogy with Grothendieck's algebraic de Rham forms.  We thus get a commutative
{\em differential} graded algebra (a CDGA) $(\Lambda^\bullet
\Omega^1_\scC, \wedge, d)$.  We refer to this CDGA as the {\sf
  $\cin$-algebraic de Rham complex} of the $\cin$-ring $\scC$. 
  
Note that in the case where the $\cin$-ring $\scC$ is the $\cin$-ring of
smooth functions $\cin(M)$ on a manifold $M$ the CDGA
$\Lambda^\bullet \Omega^1_{\cin(M)}$ of $\cin$-algebraic differential
forms  is the CDGA $\bOmega^\bullet_{dR}(M)$ of ordinary de Rham
differential forms on the manifold $M$ (see \cite{LdR}).   This is not
completely obvious.

For any module $\scM$ over an $\R$-algebra $A$ a homomorphism
$\varphi:\scM =\Lambda^1\scM \to A =\Lambda^0\scM$ extends uniquely to a degree -1 derivation $
\varphi^\bullet:\Lambda^\bullet \scM\to \Lambda^{\bullet -1} \scM $ of the
CGA $(\Lambda^\bullet \scM, \wedge)$.   Consequently there is an
isomorphism of $A$-modules
\[
 \Hom(\scM, A) \xrightarrow{\, \simeq \,} \Der^{-1} (\Lambda^\bullet \scM).
\]  
Since for any $\cin$-ring $\scC$ the module of $\cin$-derivations
$\cin\Der(\scC)$ is isomorphic, as an $\scC$-module, to
$\Hom(\Omega^1_\scC, \scC)$ there is a canonical injective map
\[
\iota_\scC: \cin\Der(\scC) \to \Der^{-1} (\Lambda^\bullet\Omega^1_\scC),
\]
the {\sf contraction map}. 
The map $\iota_\scC$ is a map of $\scC$-modules.

In differential geometry Lie derivatives of differential forms with
respect to vector fields are
usually defined in terms of flows of the vector fields.  Given a
derivation of a 
$\cin$-ring  it is not clear where or how to integrate this
derivation to a flow.   Thanks to the work of Dubuc, there is a
spectrum functor $\Spec$ for $\cin$-rings that sends them to
$\cin$-ringed spaces.  Unlike the case of
commutative rings the functor $\Spec$  is not
terribly well behaved.  For example there are  nonzero $\cin$-rings
with the same spectrum as $\Spec(0)$, the spectrum of the zero
ring. In particular their set of points is empty.   And even when the $\cin$-ring $\scC$ in question is
the ring of functions on the space underlying $\Spec(\scC)$ and one does know how to integrate a
derivation of $\scC$ to a flow, it may well happen that different
derivations have exactly the same flows.  (The example I have in mind
here is the standard Cantor set $C \subset \R$.  The vector field  $\frac{d}{dx}$
on the real line
``restricts'' to a nontrivial  derivation  $\frac{d}{dx} |_C$ on
$\cin(C):= \cin(\R)|_C$ but its flow only
exists for time 0.  Moreover, for any smooth function $f\in \cin(\R)$
the restriction $(f\frac{d}{dx}) |_C$ has exactly the same trivial
flow as  $\frac{d}{dx} |_C$.) For these reasons 
given  a derivation $v\in \cDer(\scC)$ we   {\em define}
the Lie derivative $\cL_\scC(v): \Lambda^\bullet \Omega^1_\scC  \to
\Lambda^\bullet \Omega^1_\scC $ by the Cartan magic formula:
\[
\cL_\scC(v) := \iota_\scC(v) \circ d + d\circ \iota_\scC(v).  
\]
The Lie derivative $\cL_\scC(v)$ is a degree 0 derivation of the CGA
$(\Lambda^\bullet\Omega^1_\scC, \wedge)$  for every $v\in \cin\Der(\scC)$.
We thus obtain an $\R$-linear map
\[
\cL_\scC: \cin\Der(\scC) \to \Der^{0} (\Lambda^\bullet\Omega^1_\scC),
\qquad v\mapsto \cL_\scC(v) = \iota_\scC(v) \circ d + d\circ \iota_\scC(v),
\]  
which we call the {\sf Lie derivative map. }    We are now in position to formulate  Cartan
calculus on a $\cin$-ring.  We drop the subscript ${\, }_\scC$ to reduce
the clutter.

\begin{theorem} \label{prop:cartan_calc_level0}
For any $\cin$-ring $\scC$ the contraction $\iota: \cin\Der(\scC) \to \Der^{-1}
(\Lambda^\bullet\Omega^1_\scC)$   and the Lie derivative $\cL:  \cDer(\scC)
\to \Der^{0} (\Lambda^\bullet\Omega^1_\scC)$ maps, which were
described above,  satisfy the
following identities: 
\begin{align*}
&(\textrm{i})\quad \cL(v) \circ d = d \circ \cL(v) \qquad \qquad \qquad \qquad (\textrm{ii}) &\cL([v,w]) = [\cL(v), \cL(w)]\\
&(\textrm{iii})\quad[\cL(v), \iota(w)] = \iota([v,w]) \qquad \qquad \qquad (\textrm{iv}) &[\iota(v),
                                                      \iota(w)] = 0\\
&(\textrm{v}) \quad \cL(v) = [d, \iota(v)]
\end{align*}
for all $v,w\in \cDer (\scC)$.
\end{theorem}

The Cartan calculus of Theorem~\ref{prop:cartan_calc_level0} 
does agree with the usual Cartan calculus on  manifolds.   As was
mentioned previously when $\scC = \cin(M)$ for a smooth manifold $M$,
the CDGA 
$\Lambda^\bullet \Omega^1_{\cin(M)}$ of $\cin$-algebraic differential
forms  is the CDGA $\bOmega^\bullet_{dR}(M)$ of ordinary de Rham
differential forms on the manifold $M$, the Lie algebra
$\cin\Der(\cin(M)$ is the Lie algebra of vector fields $\chi(M)$ on $M$, the
contraction map $\iota:\cin\Der (\cin(M)) \to \Der^{-1}
(\Lambda^\bullet\Omega^1_{\cin(M)})$ is the usual  contraction map $\imath:
\chi(M)\to \Der^{-1}(\bOmega_{dR}^\bullet)$  between vector fields and
differential forms, and so on. \\[6pt]
Given a local $\cin$-ringed space $(M, \scA)$ it may be tempting to
construct the corresponding Cartan calculus as follows: let $\scC=
\scA(M)$, the $\cin$-ring of global sections of the structure sheaf
$\scA$.
We then declare the real vector space $\cDer (\scC)$ with the
commutator bracket to be the Lie algebra of vector fields on
$(M,\scA)$.  In the same spirit we declare the commutative
differential graded algebra $(\Lambda^\bullet \Omega ^1_{\scA(M)},
\wedge, d)$ to be the CDGA of differential forms.  This, however, is
probably too naive.  In particular it does not capture the locality of
Cartan calculus on manifolds.   On a manifold $M$ vector
fields and differential forms are  sections the appropriate
vector bundles, and consequently, are sections of sheaves.  For
a general $\LCRS$ $(M,\scA)$ the presheaf of $\cin$-K\"ahler differentials
$\Omega^1_\scA$, defined by
\[
U\mapsto \Omega^1_{\scA(U)}
\]  
for all open sets $U\subset M$, is not necessarily a sheaf.  And even
when it is a sheaf, there is no reason for its exterior powers
$\Lambda^k \Omega^1_\scA$ to be sheaves.  So the ``right'' object to
consider as a replacement of the de Rham complex on a manifold is the
CDGA $(\bOmega_\scA^\bullet , \wedge, \bd)$ obtained by sheafification of the presheaf
\[
U\mapsto (\Lambda^\bullet \Omega^1_{\scA(U)}, \wedge , d),
 \] 
cf.\ \cite{LdR}. But then there is no reason for the global sections of the presheaf
$\Lambda^k \Omega ^1_{\scA}$ and its sheafification $\bOmega ^k_{\scA}$ to be equal.

The situation with derivations is worse.  Except for some
cases (such as differential spaces, see Appendix~\ref{app:A}), given a
local $\cin$-ringed space $(M,\scA)$ and two open subsets $V, U$ of
$M$ with $V\subset U$ there is no evident ``restriction'' map
\[
\cDer(\scA(U)) \to \cDer(\scA(V)).
\]  
Recall that even in the case where $M$ is a manifold and $\scA $ is
the sheaf $\cin_M$ of smooth functions on $M$  the construction of
such a restriction requires a bit of work.   It can be built directly
using bump functions or indirectly by identifying $\cDer (\cin_M(U)) =
\cDer (\cin(U))$ with the sections of the restricted tangent bundle
$TM|_U\to U$.

The solution to these difficulties, I believe, is to reformulate
Cartan calculus in terms of sheaves.  This means, in particular, that
the calculus needs to be stated in terms of maps rather than in terms
of elements, as it is traditionally done and as we did in
Theorem~\ref{prop:cartan_calc_level0}.\\[6pt]
Just as in ordinary algebraic geometry, given two presheaves $\scM$, $\scN$ of $\scA$-modules over
a local $\cin$-ringed space $(M,\scA)$ there is the so called hom
presheaf $\cHom(\scM, \scN)$ of $\scA$-modules; it is defined by
\[
  \cHom(\scM, \scN) (U): = \Hom(\scM|_U, \scN|_U)
\]
for all open subsets $U\subset M$.  Note that if $\scN$ is a sheaf
then $ \cHom(\scM, \scN)$ is a sheaf as well (the standard exercise in
algebraic geometry assumes that both $\scM$ and $\scN$ are sheaves,
but the assumption on $\scM$ is not necessary). As was mentioned above, given a
$\LCRS$ $(M,\scA)$ we have a sheaf of CDGAs
$(\bOmega_\scA^\bullet, \wedge, \sh(d))$ of $\cin$-algebraic de Rham
differential forms on $M$ which  is obtained by
sheafifying the presheaf
$
U\mapsto (\Lambda^\bullet \Omega^1_{\scA(U)}, \wedge, d).
$
There is also a sheaf of derivations $\cin\Dert(\scA)$, which is
isomorphic to the sheaf of $\scA$-modules $\cHom(\Omega^1_\scA, \scA)
\simeq \cHom(\bOmega^1_\scA,\scA)$ ($\bOmega^1_\scA$
denotes the sheafification of the presheaf $\Omega^1_\scA$ of $\cin$-K\"ahler
differentials of the sheaf $\scA$).  The
sheaf $\cin \Dert(\scA)$ is a sheaf of real Lie algebras.

On the other hand, for any sheaf $(\scM^\bullet, \wedge)$ of graded
commutative algebras over $(M,\scA)$ there is a graded presheaf
$\Dert^\bullet (\scM^\bullet)$ of graded derivations of
$(\scM^\bullet, \wedge)$.  The presheaf $\Dert^\bullet (\scM^\bullet)$
is a presheaf of real graded Lie algebras.  It is not hard to check
that the contraction map
$\iota_\scC: \cin\Der(\scC) \to \Der^{-1}
(\Lambda^\bullet\Omega^1_\scC)$ is natural in the ring $\scC$.
Consequently it can be upgraded to a map of presheaves
\[
\bi: \cin\Dert(\scA)\to \Dert^{-1} (\Lambda^\bullet
\Omega^1_\scA )
\]
and then, using the fact that sheafification is a
functor, to a map 
\[
\fri   :\cin\Dert(\scA)\to \Dert^{-1} (\bOmega^\bullet_\scA)
\]
of presheaves, which  is our {\sf contraction map}.

The Lie derivative map $\cL_\scC: \cin\Der(\scC) \to \Der^{0}
(\Lambda^\bullet\Omega^1_\scC)$ is also natural in the $\cin$-ring
$\scC$.  So it too can be upgraded to a map of presheaves
\[
\scL: \cin\Dert(\scA)\to \Dert^{0} (\Lambda^\bullet
\Omega^1_\scA )
\]
and then to the map 
\[
\frL :\cin\Dert(\scA)\to \Dert^{0} (\bOmega^\bullet_\scA),
\]
 the {\sf Lie derivative} map.
The exterior derivative $\sh(d)\in \Der^1 (\bOmega^\bullet_\scA)$ (the
sheafification of $d\in \Der^1(\Lambda^\bullet \Omega_\scA)$) gives rise
to a map of (graded) presheaves
\[
  {ad}(\sh(d)): \Dert^\bullet (\bOmega^\bullet_\scA) \to
  \Dert^{\bullet +1} (\bOmega^\bullet_\scA)
\]  
With these
definitions/constructions in place we can state the  main result of the paper:

\begin{theorem} \label{main_thm}  Let $(M,\scA)$ be a local $\cin$-ringed space,
  $\cin\Dert(\scA)$ the sheaf of $\cin$-derivations of the sheaf
  $\scA$ and $\Dert^\bullet (\bOmega_\scA^\bullet ) $ the presheaf of
  graded Lie algebras of graded derivations of the sheaf
of de Rham (differential) forms,
  $ad(\sh(d)): \Dert^\bullet (\bOmega^\bullet_\scA) \to
  \Dert^{\bullet +1} (\bOmega^\bullet_\scA)$ the corresponding map between
  graded derivations,
  $\bfi :\cin\Dert(\scA)\to \Dert^{-1}
  (\bOmega^\bullet_\scA)$ the contraction map and
  $ \frL :\cin\Dert(\scA)\to \Dert^{0} (\bOmega^\bullet_\scA)$
  the Lie derivative map.
Then the following equalities hold:
  \begin{eqnarray}
    ad(\sh(d))\circ \bfi &= &  \frL \label{eq:5.9j}\\
  ad(\sh(d))\circ  \frL &=&0 \label{eq:5.10j}\\
   \frL \circ [\cdot, \cdot] &= &[\cdot, \cdot] \circ ( \frL\times
  \frL)
                                  \label{eq:5.11j}\\
 \, [\cdot, \cdot] \circ ( \frL\times \bfi) &= & \bfi \circ   [\cdot,
  \cdot] 
                                                 \label{eq:5.12j}\\
\, [\cdot, \cdot] \circ ( \bfi \times\bfi)
                                          &=& 0.
                                              \label{eq:5.13j}
\end{eqnarray}    
Here and elsewhere $[\cdot, \cdot]$ denotes both the Lie bracket on
$\cin\Dert(\scA)$ and the (graded) bracket on $ \Dert^{\bullet }
(\bOmega^\bullet_\scA)$, depending on the context.  We think of both brackets
as maps of sheaves.
\end{theorem}

\subsection*{Organization of the paper}
In Section~\ref{sec:2} we develop 
some mathematical background.  We start by recalling the definition of
a $\cin$-ring and proceed to discuss two classes of local
$\cin$-ringed spaces --- the differential spaces in the sense of
Sikorski and $\cin$-schemes of Dubuc.  This is not new, although
differential spaces and $\cin$-schemes are usually not discussed
together.  For a local $\cin$-ringed space $(M,\scA)$ we define the
$\scA(M)$ module $\cin\Der(\scA))$ of $\cin$-derivations of the
structure sheaf $\scA$ and prove that it is a real Lie algebra under
the commutator of maps of sheaves.  We then develop presheaves of
commutative graded algebras (CGAs) and commutative differential graded
algebras (CDGAs) over local $\cin$-ringed spaces.  Our main example of
such presheaves are sheaves of $\cin$-algebraic de Rham (i.e., differential) forms
introduced in \cite{LdR}.  We study presheaves of graded derivations
of presheaves of CGAs.  We then turn to internal homs and, in
particular, introduce the presheaf $\cin\Dert(\scA)$ of
$\cin$-derivations of a local $\cin$-ringed space $(M, \scA)$.  It is
a sheaf of real Lie algebras. The section ends with a proof that
presheaf of graded derivations of a presheaf of commutative graded
algebras over a local $\cin$-ringed spaces is a presheaf of real
graded Lie algebras.  In Section~\ref{sec:ring} we prove
Theorem~\ref{prop:cartan_calc_level0}.  In Section~\ref{sec:sheaf} we
prove Theorem~\ref{main_thm}.  In Section~\ref{sec:examples} we
compute examples of Lie algebras  of derivations of smooth functions
of differential spaces and show that while in one class of cases
(continuous functions on a topological space) the Lie algebra of
derivations is zero, it is nonzero for closed subsets of $\R^n$ with
dense interiors and  for $M = \{(x,y)\in \R^2 \mid xy =0\}$.

In Appendix~\ref{app:A} we prove that for a differential space $(M,
\scF)$ the global sections functor $\Gamma: \cin\Der(\scF_M) \to  \cin\Der(\scF)$
is an isomorphism of Lie algebras, where  $\cin\Der(\scF_M)$ is the
Lie algebra of $\cin$-derivations of
the structure sheaf $\scF_M$ and $\cin\Der(\scF)$ is the Lie algebra
of $\cin$-derivations of the $\cin$-ring $\scF$.

\subsection*{Acknowledgements}

I thank Yael Karshon for many many conversations about differential
spaces, $\cin$-rings and related mathematics.  I thank Casey Blacker
for a number of useful comments.\\

\noindent
This material is based upon work {partially} supported by the Air Force
Office of Scientific Research under award number FA9550-23-1-0337. 

\section{Background} \label{sec:2} In this section we develop some
of the mathematics necessary to prove
Theorem~\ref{prop:cartan_calc_level0} and Theorem~\ref{main_thm}.  We
start with some fairly well known mathematics ($\cin$-rings,
differential spaces in the sense of Sikorski and $\cin$-schemes) and proceed to
presheaves of CGAs and CDGAs over local $\cin$-ringed spaces.

To recall the definition of a $\cin$-ring we 
need to first define the category $\Euc$ of Euclidean (a.k.a.\
Cartesian) spaces.

\begin{definition}[The category $\Euc$ of Euclidean spaces]
  The objects of the category $\Euc$ are the coordinate vector spaces
  $\R^n$, $n\geq 0$,  thought of as $\cin$ manifolds.  The morphisms are
  $C^\infty$ maps.
\end{definition}
We note that all objects of $\Euc$ are finite powers of one object:
$ \R^n = (\R^1)^n$ for all $n\geq 0$.

\begin{definition}\label{def:C2}
A {\sf $C^\infty$-ring} $\cC$ is a functor
\[
\cC:\Euc \to \Set
\]
from the category $\Euc$ of Euclidean spaces to the category $\Set$ of
sets that {\em preserves finite products}.

A {\sf morphism} of $\cin$-rings from a $\cin$-ring $\cC$ to a
$\cin$-ring $\cB$ is a natural transformation $\cC\Rightarrow \cB$.
\end{definition}

\begin{remark} We unpack 
Definition~\ref{def:C2}.
\begin{enumerate}
\item For any $n\geq 0$, the set $\cC(\R^n)$ is an $n$-fold product of
  the set $\scC: = \cC(\R^1)$, i.e., $\cC(\R^n) =\scC^n$.  Moreover
  the canonical projection $\pr_i: \scC\, ^n \to \scC$
  on the $i$th factor is $\cC(\R^n\xrightarrow{x_i} \R) $, where $x_i$
  is the $i$th coordinate function.
\item For any smooth function $g= (g_1,\ldots, g_m):\R^n \to \R^m$ we
  have a map of sets
\[
\cC(g): \cC(\R^n) = \scC^n \to \scC^m = \cC(\R^m)
\]
with
\[
\pr_j \circ \cC (g) = \cC(g_j).
\]
\item For any triple of smooth functions $\R^n\xrightarrow{g} \R^m
  \xrightarrow{f} \R$ we have three maps of sets $\cC(g)$, $\cC(f)$,
  $\cC(f\circ g) $ with
\[
\cC(f) \circ \cC(g) = \cC(f\circ g).
\]
\item $\cC(\R^0)$ is a single point set $*$.  This is because by definition
  product-preserving functors take terminal objects to terminal objects.
\end{enumerate}
It follows that $\scC = \cC(\R^1)$ is a set together with, for every
$n\geq 0$, a collection of $n$-ary operations indexed by the set of smooth
functions $\cin(\R^n)$.  This is the ``definition'' of a
$\cin$-ring given on page~\pageref{page:c-ring} of the introduction.
Conversely given a set $\scC$ with a family of operations indexed by 
smooth functions there is a unique product-preserving functor $\cC:\Euc\to
\Set$ with $\cC(\R^1) = \scC$.  

Throughout the paper we will treat $\cin$-rings as sets with
operations and morphisms of $\cin$-rings as operation-preserving maps
(rather than as functors and natural transformations).  This treatment
is standard, cf.\ \cite{MR} or \cite{Joy}.
\end{remark}

\begin{remark} \label{rmrk:cin-R-alg}
Any $\cin$-ring has an underlying $\R$-algebra.  This can be
seen as follows.  The category $\Euc$ contains the subcategory $\Euc_{poly}$
with the same objects whose morphism are polynomial maps.  A
product-preserving functor from $\Euc_{poly}$ to $\Set$ is an
$\R$-algebra.  Clearly any product preserving functor $\cC:\Euc \to
\Set$ restricts to a product preserving functor from $\Euc_{poly}$ to $\Set$.

The structure of an $\R$-algebra on a $\cin$-ring $\scC$ can also be
seen on the level of sets and operations. 
Note that $\R^0 =\{0\}$, $\cin(\R^0) = \R$ and that for any set $\scC$
the zeroth power $\scC^0$ is a 1-point set $*$.  Consequently for
a  $\cin$-ring $\scC$, $n=0$, and $\lambda \in C^\infty (\R^0) = \R$
the corresponding 0-ary operation is a function $\lambda_\scC: *\to
\scC$, which we identify with an element of $\scC$.  By abuse of
notation we denote this element by $\lambda_\scC$.  This gives us a
map $\R\to \scC$, $\lambda \mapsto \lambda_\scC$.  One can show that
unless $\scC$ consists of one element (i.e., $\scC$ is the zero $\cin$-ring)
the map is injective.    Together with the operations $+:\scC^2\to
\scC$ and $\cdot: \scC^2 \to \scC$, which correspond to the functions
$h(x,y) = x+y$ and $g(x,y) = xy$, the 0-ary operations make any
$\cin$-ring into a unital $\R$-algebra.  We will not notationally
distinguish between a $\cin$-ring and the corresponding (underlying) $\R$-algebra.
\end{remark}

\begin{notation}
$\cin$-rings and their morphisms form a category that we denote by
$\cring$.
\end{notation}

\begin{example} \label{ex:2.1new}
Let $M$ be a $\cin$-manifold and $\cin(M)$ the set of
smooth (real-valued) functions. Then $\cin(M)$, equipped with the %
operations
\[
g_{\cin(M)} (a_1,\ldots, a_m) := g\circ (a_1,\ldots, a_m),
\]
(for all $m$, all $g\in \cin(\R^m)$ and all $a_1,\ldots, a_m \in \cin(M)$) is a $\cin$ ring.
\end{example}

\begin{example}
  The real line $\R$ is a $\cin$-ring:   the $m$-ary operations are given by
\[  
  g_\R(a_1,\ldots, a_m):= g(a_1,\ldots, a_m)
\]
for all $a_1,\ldots, a_m\in \R$ all $g\in \cin(\R^m)$.   Note that
this example  is consistent with Example~\ref{ex:2.1new}: if $M$ is a
point $*$, then $\cin(M) = \R$ (thought of as constant functions $*\to
\R$ and composing a function $g\in \cin(\R^m)$ with an $m$-tuple of
constant functions is the same as evaluating $g$ on the constants.
\end{example}  

\begin{example} \label{ex:2.1new+}
Let $M$ be a topological space  and $C^0(M)$ the set of
continuous real-valued  functions. Then $C^0(M)$, equipped with the %
operations
\[
g_{\cin(M)} (a_1,\ldots, a_m) := g\circ (a_1,\ldots, a_m),
\]
is also a  $\cin$ ring.
\end{example}

\begin{definition} \label{def:2.6}
We define a {\sf module} over a $\cin$-ring $\scC$ to be a module over
the $\R$-algebra underlying $\scC$ (cf.\ Remark~\ref{rmrk:cin-R-alg}).
\end{definition}  

\begin{remark}
There is a variety of reasons why Definition~\ref{def:2.6} is a
reasonable definition.  See \cite{LdR} for a discussion.
\end{remark}

\begin{example}
Any $\cin$-ring $\scC$ is a module over $\scC$.
\end{example}

\begin{example}
Let $M$ be a manifold and $E\to M$ a vector bundle. Then the set of
smooth sections $\Gamma(E,M)$ of $E$ is a module over the $\cin$-ring $\cin(M)$.
\end{example}

\begin{example} \label{ex:points-module}
Let $M$ be a manifold and $x\in M$ a point. Then the evaluation map
$ev_x: \cin(M)\to \R$ is a homomorphism of $\cin$-rings hence make
$\R$ a module over $\cin(M)$.   Note that this module structure
of $\R$ depends on the choice of a point $x$.
\end{example}

\begin{definition}[The category $\Mod$ of modules over $\cin$-rings] \label{rmrk:mod} \label{def:mod}
Modules over $\cin$-rings form a category; we denote it by $\Mod$.   The objects of
$\Mod$ are pairs $(R, M)$ where $R$ is a $\cin$-ring and $M$ is a
module over $R$.   A {\sf morphism} in $\Mod$ from $(R_1, M_1)$ to
$(R_2, M_2)$ is a pair $(f,g)$ where $f:R_1\to R_2$ is a map of
$\cin$-rings and $g:M_1\to M_2$ is a map of abelian groups so that
\[
g(r\cdot m) = f(r)\cdot g(m)
\]
for all $r\in R_1$ and $m\in M_1$.
\end{definition}

\begin{definition} \label{def:der2} Let $\scC$ be a $\cin$-ring and
$\scM$ an $\scC$-module.  A {\sf $\cin$-derivation of $\scC$ with
values in $\scM$ } is a map $X:\scC\to \scM$ so that for any $n>0$,
any $f\in \cin(\R^n)$ and any $a_1,\ldots, a_n\in \scC$
\begin{equation}\label{eq:der} X(f_\scC(a_1,\ldots,a_n)) =
\sum_{i=1}^n (\partial_i f)_\scC(a_1,\ldots, a_n)\cdot X(a_i).
 \end{equation}
\end{definition}

\begin{example}
Let $M$ be a smooth manifold. A $\cin$-derivation $X:\cin(M) \to
\cin(M)$ of the $\cin$-ring of smooth 
functions $\cin(M)$ with values in $\cin(M)$ is an ordinary vector
field.

The exterior derivative $d: \cin(M) \to \bOmega_{dR}^1(M)$ is a $\cin$
derivation of $\cin(M)$ with values in the module $\bOmega_{dR}^1(M)$ of
ordinary 1-forms.

Let $x\in M$ be a point and $\R$ a $\cin(M)$-module with the module
structure coming from the evaluation map $ev_x:\cin(M)\to \R$, cf.\
Example~\ref{ex:points-module}.  Then a derivation of $\cin(M)$ with
values in the module $\R$ is a tangent vector at the point $x\in M$.
\end{example}

\begin{notation}
The collection of all $\cin$-derivations of a $\cin$-ring $\scC$ with values
in an $\scC$-module $\scM$ is a $\scC$-module.  We denote it by
$\cDer(\scC, \scM)$.  In the case where $\scM = \scC$ we write
$\cDer(\scC)$ for $\cDer(\scC, \scC)$.
\end{notation}

\begin{remark}
A $\cin$-ring $\scC$ has an underlying $\R$-algebra structure.
Therefore given a $\scC$-module $\scM$ we may also consider the space
$\Der(\scC, \scM)$ of $\R$-algebra derivations.  It is easy to see
that any $\cin$-ring derivation of $\scC$ with values in $\scM$ is an
$\R$-algebra derivation (apply a $\cin$-ring derivation to the
operation defined by $f(x,y) = xy\in \cin(\R^2)$ to get the
product rule, for example).   In general $\Der(\scC, \scM)$ is
strictly bigger than $\cin\Der(\scC, \scM)$.  For example,  Osborn shows in \cite[Corollary to Proposition 1]{Osborn} that if $\scC = \cin(\R^n)$ and 
$\scM =\Omega^1_{\cin(\R^n), alg}$ is the module of ordinary  K\"ahler differentials (in the sense of commutative algebra), then 
the universal derivation $d: \cin(\R^n) \to \Omega^1_{\cin(\R^n), alg}$ is {\em not} a $\cin$-derivation
(it is only a $\cin$-derivation on algebraic functions).   
However when the $\cin$-ring in question is a differential structure
$\scF$ on some differential space (see Definition~\ref{def:sikorski}
below) and $\scM = \scF$, then
$\cin\Der(\scF, \scF) = \Der(\scF, \scF)$.  This follows from a
theorem of Yamashita cited above because the $\cin$-ring $\scF$ is
point-determined (see Definition~\ref{def:pd} and Example~\ref{ex:ds_pd}).
\end{remark}

\begin{lemma} \label{lem:cder-lie}
Let $\scC$ be a $\cin$-ring, $X,Y\in \cDer(\scC)$ two
$\cin$-derivations.  Then their commutator
\[
[X,Y]:= X\circ Y - Y\circ X
\]
is also a $\cin$-derivation of $\scC$.  Consequently $(\cDer(\scC),
[\cdot, \cdot])$ is a real Lie algebra.
\end{lemma}  
\begin{proof}
This is a computation.  Suppose $n>0$, $a_1,\ldots, a_n\in \scC$ and
$f\in \cin(\R^n)$.  Then 
\[
\begin{split}
  X(Y(f_\scC(a_1,\ldots, a_n)) - Y(X(f_\scC(a_1,\ldots, a_n)) =\\
  X(\sum_i(\partial_if)_\scC(a_1,\ldots, a_n) Y(a_i) ) -   Y(\sum(\partial_jf)_\scC(a_1,\ldots, a_n) X(a_j) )\\ 
 =   \sum_i X((\partial_if)_\scC(a_1,\ldots, a_n) )Y(a_i) +
 \sum_i(\partial_if)_\scC(a_1,\ldots, a_n) X( Y(a_i)) \\- \sum_jY((\partial_jf)_\scC(a_1,\ldots, a_n)) X(a_j)  -\sum_j(\partial_jf)_\scC(a_1,\ldots, a_n) Y(X(a_j) ).
\end{split}
\]  
Now
\[
\begin{split}
  \sum_iX((\partial_if)_\scC(a_1,\ldots, a_n) )Y(a_i)
  -\sum_jY((\partial_jf)_\scC(a_1,\ldots, a_n)) X(a_j) \\
  = \sum_{i,j} (\partial^2_{ji} f) _\scC  (a_1,\ldots, a_n) X(a_j) Y(a_i) -
 \sum_{i,j} (\partial^2_{ij} f) _\scC  (a_1,\ldots, a_n)Y(a_i)  X(a_j) = 0
\end{split}
\]
and the result follows.
\end{proof}

\begin{definition}
 A {\sf module of $\cin$-K\"ahler
differentials} over a $\cin$-ring $\scC$ is an $\scC$-module $\Omega^1_\scC$
together with a $\cin$-derivation $d_\scC:\scC\to
\Omega^1_\scC$, called the {\sf universal derivation}, with the following  property:  for any $\scC$-module $\scN$ and any derivation
$X:\scC\to \scN$ there exists a unique map of modules $\varphi_X:
\Omega^1_A\to \scN$ with
\[
\varphi_X \circ d = X.
 \] 
\end{definition}

\begin{remark}
The module of $\cin$-K\"ahler differentials is not be confused with the module of ordinary algebraic K\"ahler differentials, since the two modules are quite different.
For example G\'omez \cite{Gomez} shows that the cardinality of any set
of generators of the module of algebraic K\"ahler differentials $\Omega^1_{\cin(\R^n), alg}$ of $\cin(\R^n)$ is at least the cardinality of the reals.  Compare this with the fact that the module of $\cin$-K\"haler differentials $\Omega^1_{\cin(\R^n)}$ is isomorphic, as a $\cin(\R^n)$ module to the module of 
de Rham 1-forms $\bOmega^1 (\R^n)$, which is a free $\cin(\R^n)$ module of rank $n$ --- see \cite{LdR}.
\end{remark}

\begin{theorem}[Dubuc and Kock\cite{DK}] \label{thm:DK}
For any $\cin$-ring $\scC$ there exists a module of K\"ahler
differentials $\Omega_\scC^1$ over $\scC$ together with a universal $\cin$-derivation
$d_\scC:\scC\to \Omega_\scC^1$.
\end{theorem}
\begin{proof}
See \cite{DK} for a proof when $\scC$ is an algebra of a Fermat
theory.  A shorter proof in the case when $\scC$ is a $\cin$-ring can be found
in \cite{Joy}.    
\end{proof}

\begin{remark}
One can show that the $\scC$-module $\Omega^1_\scC $ of K\"ahler
differentials of a $\cin$-ring is generated, as a module, by the set
$\{d_\scC a\}_{a\in \scC}$.   This is easy to see if one uses the
construction of $\Omega^1_\scC$ in \cite{Joy}.
\end{remark}  
\begin{remark} \label{rmrk:5.12}
If $M$ is a manifold then the $\cin(M)$-module $\bOmega^1_{dR}(M)$ of
ordinary one-forms on $M$ is the module of K\"ahler differentials
$\Omega_{\cin(M)}^1$ and the exterior derivative $d:\cin(M)\to
\bOmega^1_{dR}(M)$ is the universal derivation.  This fact is stated in
\cite[Example~5.4]{Joy}  without proof.   In the more general case where $M$ is a
manifold with corners this fact is proved in \cite[Proposition~4.7.5]{FS}.
\end{remark}

\begin{remark} \label{rmkr:cder-as-dual}
The universal property of the derivation $d:\scC\to
\Omega^1_\scC$ implies that the map
\[
 d^*: \Hom (\Omega ^1_\scC, \scC) \to \cDer (\scC), \qquad d^*
 (\varphi) := \varphi \circ d   
\]
is a bijection.
It is easy to see that $d^*$ is a homomorphism of $\scC$-modules, hence
an isomorphism of $\scC$-modules.
\end{remark}

\subsection*{Graded objects}
\begin{remark}
There are two conventions in the literature regarding ($\N$- or $\Z$-)
graded objects in an abelian category: they can be viewed either as
direct sums or as sequences of objects. Kemps and Porter
\cite[p.~207]{KP}, for instance, use the second convention.

Depending on the category one is working in the choice of the
convention may or may not matter. For example, a $\Z$-graded real
Lie algebra $\fg$ can be viewed as a direct sum
$\fg=\oplus_{n\in \Z} \, \fg_n$ together with an $\R$-bilinear map
(the bracket) $\fg\times \fg\to \fg$ satisfying the appropriate
conditions.  It can equivalently be viewed as a sequence of vector
spaces $\{\fg_n\}_{n\in \Z}$ together with a family of linear maps
$\{\fg_n \otimes \fg_m \to \fg_{n+m}\}_{n,m\in \Z}$ (again, subject to
conditions).  So for graded
Lie algebras the two conventions are equivalent.

However, in the case of graded sheaves of modules, the two approaches are
not equivalent: see  Remark~\ref{rmrk:graded_sheaves}.
Since we will need the (pre-)sheaf versions of various graded objects,
we use the second convention throughout: a graded object is a sequence
of objects.  Since graded derivations may have negative degrees, our
graded objects are graded by the integers $\Z$.  Note that any object
graded by the natural numbers $\N$ can be viewed as a $\Z$-graded
object by placing zeros in negative degrees.
\end{remark}

The following definition is a variation on  the conventional definition of
a commutative graded algebra.   
\begin{definition} A {\sf commutative graded algebra} (CGA) over a
  $\cin$-ring $\scC$ is a sequence of $\scC$-modules $\{B^k\}_{k\in \Z}$
  together with a sequence of $\scC$-bilinear maps $\wedge: B^k\times
  B^\ell \to B^{k+\ell} $ (or, equivalently of $\scC$-linear maps   $\wedge: B^k\otimes_{\scC}
  B^\ell \to B^{k+\ell} $; we don't notationally distinguish between
  the two) so that $\wedge $ is associative and graded
  commutative. That is for all $k,\ell, m\in \Z$ and all $a\in B^k$,
  $b\in B^\ell$ and $c\in B^m$
  \begin{enumerate}
  \item $(a\wedge b) \wedge c = a\wedge (b\wedge c)$ and
  \item $b\wedge a = (-1)^{k\ell} a \wedge b$ (Koszul sign convention).
  \end{enumerate}
  We denote such a CGA by $(B^\bullet, \wedge)$. We write $b\in
  B^\bullet$ if $b\in B^k$ for some $k\in \Z$.
  We refer to elements of $B^k$ as {\em homogeneous elements of degree
    k} and write $|b| =k$ to indicate that $b\in B^k$.
\end{definition}

\begin{example} \label{ex:exterior}
Let $\scC$ be a $\cin$-ring and  $\scM$ an $\scA$-module. Then the
exterior algebra $\Lambda ^\bullet \scM =\{ \Lambda^k \scM\}_{k\in \Z}$
  together with the usual wedge product $\wedge$ 
  is a CGA over $\scC$.  Note that the exterior powers are taken over
  $\scC$. Note also that $\Lambda^0 \scM = \scA$ and that
  $\Lambda^k \scM=0$ for $k<0$ by convention.
\end{example}  
\begin{remark}
Commutative graded algebras (CGAs) are really graded-commutative algebras:
that is, they are graded and they are commutative up to the appropriate sign;
 they are not commutative on the nose.    This is why sometimes they are also 
called GCAs.
\end{remark}

\begin{definition}[The category $\CGA$ of commutative graded algebras] \label{def:catCGA}
Recall that modules over $\cin$-rings form a category $\Mod$ (Definition~\ref{def:mod}).
Consequently commutative graded algebras over $\cin$-rings form a
category as well.  We denote it by $\CGA$.  The objects of this category are pairs $(\scC,
(\scM^\bullet, \wedge))$,  where $\scC$ is a $\cin$-ring and
$(\scM^\bullet, \wedge)$ is a CGA over $\scC$.  A morphism from $(\scC_1,
(\scM_1^\bullet, \wedge))$ to $(\scC_2,
(\scM_2^\bullet, \wedge))$ in $\CGA$ is a pair $(\varphi, f
= \{f^k\}_{k\in \Z})$
where, for each $k$,  $(\varphi, f^k): (\scC_1, \scM_1^k)\to (\scC_2, \scM_2^k)$
is a morphism in the category $\Mod$ of modules and, additionally, $f$ preserves the
multiplication: $f(m\wedge m') = f(m)\wedge f(m')$.
\end{definition}

\begin{definition} \label{def:gr-der}
Let $(B^\bullet, \wedge)$ be a commutative graded
algebra over a $\cin$-ring $\scC$.  A {\sf graded derivation of degree
  $k\in \Z$} is a sequence of  $\R$-linear
maps 
$X= \{X^\ell :B^\ell \to B^{\ell+k}\}_{\ell\in \Z}$  so that for all homogeneous
elements $x,y\in B^\bullet$
\[
X(x\wedge y) = X(x) \wedge y + (-1)^{|x|k} x \wedge X(y),
\]
where $|x|$ is the degree of $x$, i.e., $x\in B^{|x|}$ (to reduced the
clutter the $X$ on the left stands for $X^{|x|+|y|}$ and so on). As
always our signs follow the Koszul sign convention.
\end{definition}

\begin{example}
Let $M$ be a smooth manifold, $\scC = \cin (M)$, the $\cin$-ring of
smooth functions.  Then the exterior derivative
$d:\bOmega_{dR}^\bullet(M)\to \bOmega_{dR}^{\bullet +1}(M)$ is a graded derivation
of degree +1 of the CGA $(\bOmega_{dR}^\bullet (M), \wedge)$ of global de
Rham differential forms.

Any vector field $v:M\to TM$ give rise to a derivation $\imath_v:
\bOmega_{dR}^\bullet(M)\to \bOmega_{dR}^{\bullet -1}(M)$ of degree -1:
it is the
contraction with $v$. The ordinary Lie derivative $\cL_v: \bOmega_{dR}^\bullet(M)\to
\bOmega_{dR}^{\bullet}(M)\,$ with respect to the vector field $v$ is a
degree 0 derivation of $(\bOmega_{dR}^{\bullet}(M), \wedge)$.

\end{example}

The following lemma will be useful when working with contractions.
\begin{lemma} \label{lem:1level0}
  Let $(\scM^\bullet, \wedge)$ be a CGA over a $\cin$-ring $\scC$ with
  $\scM^k = 0$ for all $k<0$ and with $\scM^0 = \scC$. Any degree -1
  derivation $X:\scM^\bullet \to \scM^{\bullet-1}$  is $\scC$-linear.
\end{lemma}  
\begin{proof}
For all $a\in \scC =\scM^0$ and all $m\in
  \scM^\bullet$
\[
X(a\wedge m) = X(a) \wedge m + (-1)^{0\deg (m)} a \wedge X(m) = a
\wedge X(m)
\]
since $X(a) \in \scM^{-1}= 0$.
\end{proof}

\begin{remark}
  The set of all graded derivations of degree $k$ of a CGA
  $(B^\bullet, \wedge)$ over a $\cin$-ring $\scC$ is a vector space
  over $\R$ and an $\scC$-module.  The operations are defined
  pointwise.  For example, if $X=\{X^\ell\}_{\ell \in \Z} $ is a
  derivation and $a\in \scC$, $b\in B^\bullet $ then
  $(aX) (b) := a(X (b)) $ and so on.
\end{remark}

\begin{notation} \label{not:gr-der}
We denote  by
$\Der^k(B^\bullet)$ {\sf the $\scC$-module of all graded derivations of degree $k$ of a
CGA $(B^\bullet, \wedge)$ over a $\cin$-ring $\scC$}.  We write $X:B^\bullet \to B^{\bullet+k}$
if  $X= \{X^\ell\}_{\ell \in \Z} \in \Der^k(B^\bullet)$ is a graded
derivation of degree $k$.
\end{notation}

\begin{definition}
  A commutative {\sf differential } graded algebra (CDGA) over a $\cin$-ring
  $\scC$ is a commutative graded algebra $(\scB^\bullet, \wedge)$
  together with a degree +1 derivation $d \in \Der^1(\scB^\bullet)$
  called a {\sf differential}, which squares to 0. That is, 
  $d\circ d =0$.

  We write $(\scB^\bullet, \wedge, d)$ to denote such a CDGA.
\end{definition}

Let $\Omega^1_\scC$ be the module of
  $\cin$-K\"ahler differentials of a $\cin$-ring $\scC$.  By
  Example~\ref{ex:exterior} the exterior algebra
  $\Lambda^\bullet\Omega^1_\scC := \{ \Lambda^k\Omega^1_\scC\}_{k\in \Z}$
 is a commutative graded algebra over the $\cin$-ring $\scC$.  We can
 do better:

\begin{theorem}\label{thm:6.2}
Let $\scC$ be a $\cin$-ring.  The universal differential $d_\scC:\scC
\to \Omega^1_\scC$ extends to unique degree 1 graded derivation
\[
d: \Lambda^\bullet\Omega^1_\scC \to \Lambda^{\bullet+1}\Omega^1_\scC 
\]
so that for all $k>0$ and all $a, b_1,\ldots, b_k\in \scC$
\begin{equation} \label{eq:6.d}
d(a \, d_\scC b_1\wedge  \ldots \wedge d_\scC b_k) = d_\scC a \wedge d_\scC b_1\wedge  \ldots \wedge d_\scC b_k.
\end{equation}
Consequently $d\circ d =0$.
\end{theorem}

\begin{proof}
See \cite[Theorem~7.14 and Corollary~7.16]{LdR}.
\end{proof}
\begin{definition} \label{def:dR_for_ring}
Theorem~\ref{thm:6.2} implies that 
for any $\cin$-ring
$\scC$ we have the commutative differential graded algebra
$(\Lambda^\bullet\Omega^1_\scC , \wedge, d)$.   We call this CDGA the
{\sf $\cin$-algebraic de Rham complex } (or
a {\sf complex of $\cin$-algebraic differential forms}) of the $\cin$-ring $\scC$.
\end{definition}

Completely analogous to the category $\CGA$ of commutative graded
algebras, commutative differential graded algebras also form a
category.

\begin{definition}[The category $\CDGA$ of commutative differential
  graded algebras] \label{def:catCDGA}  The objects of the category
  $\CDGA$ of commutative differential graded algebras are  triples 
  $(\scC, (\scM^\bullet, \wedge), d)$, where $\scC$ is a $\cin$-ring,
  $(\scM^\bullet, \wedge)$ is a CGA over $\scC$ and $d:\scM^\bullet
  \to \scM^\bullet$ is a degree 1 derivation of   $(\scM^\bullet,
  \wedge)$.  Equivalently the objects of $\CDGA$ are pairs   $(\scC,
  (\scM^\bullet, \wedge, d))$, where $\scC$ is a $\cin$-ring and $
  (\scM^\bullet, \wedge, d))$ is a CDGA over $\scC$.

  A morphism from  $(\scC_1, (\scM_1^\bullet, \wedge), d)$ to
  $(\scC_2, (\scM_2^\bullet, \wedge), d))$ is a morphism $(\varphi,
  f): (\scC_1, (\scM_1^\bullet, \wedge)) \to (\scC_2, (\scM_2^\bullet,
  \wedge))$ of commutative graded algebras with the additional
  property that $f$ commutes with the differentials.  
\end{definition}

\begin{definition}[Graded Lie algebra $\Der^\bullet (\scM^\bullet)$ of
  graded derivations of 
  a commutative graded algebra  $(\scM^\bullet, \wedge)$  over a
$\cin$-ring $\scC$] \label{def:gr-Lie-graded-der-ring}\mbox{}\\
Let $(\scM^\bullet, \wedge)$ be a  CGA over a
$\cin$-ringe $\scC$.  The for each $k\in \Z$ we have the
real vector space $\Der^k (\scM^\bullet) = \Der^k (\scM^\bullet,
\wedge)$ of graded derivations of degree $k$
(Notation~\ref{not:gr-der}).
Given two derivations $X\in \Der^k (\scM^\bullet)$ and $Y\in \Der^\ell
(\scM^\bullet)$ their graded commutator $[X,Y]$ makes sense:
\begin{equation} \label{eq:gr-comm}
\left\{([X,Y])^n := X^{\ell + n} \circ Y^n - (-1)^{k\ell}Y ^{k+n} \circ X^n:
\scM^n \to \scM^{n+k+\ell}\right \}_{n\in \Z}.
\end{equation}
One checks that the graded commutator $[X,Y]$ is a derivation of
degree $|X|+|Y| = k+\ell$ and that the pair 
\[
(\Der^\bullet (\scM^\bullet), [\cdot, \cdot] ) \equiv( \{\Der^k(\scM^\bullet)\}_{k\in \Z},
[\cdot, \cdot] )
\]
is a graded Lie algebra.
\end{definition}

\begin{notation} Given two modules $\scM, \scN$ over the same $\cin$-ring
  $\scC$ we denote the collection of maps of modules from $\scM$ to
  $\scN$ by $\Hom(\scM, \scN)$.  Note that $\Hom(\scM, \scN)$ is,
  again, a module over $\scC$.
\end{notation}
  
Our next step is to define local $\cin$-ringed spaces and (graded) modules over
local $\cin$-ringed spaces.

\begin{remark} \label{rmrk:values}
There are two ways to view
sheaves on a space $M$ with values in an algebraic category $\sfC$.
We can view them as functors $\op{\Open(M)}\to \sfC$, where $\Open(M)$
is the poset of open subsets of $M$ ordered by inclusion.   Or we can
view them as ``$\sfC$-objects'' internal to the category of set-valued
sheaves on $M$.  The latter usually amounts to having a collection of
maps of set-valued sheaves subject a number of commutative diagrams.
We will go back and forth between the two points of view.
\end{remark}

In order to state and explain the definition of {\em local} $\cin$-ringed
spaces we need to recall the definition of an {\em $\R$-point } of a $\cin$-ring.
\begin{definition}
An  {\sf $\R$-point} of a $\cin$-ring $\scC$ is a nonzero homomorphism
$p:\scC\to \R$ of $\cin$-rings.
\end{definition}  

\begin{example}
Let $M$ be a manifold and $\scC = \cin(M)$ the $\cin$-ring of smooth
functions on $M$.  Then for any point $x\in M$ the evaluation map at
$x$
\[
ev_x:\cin(M)\to \R,\qquad f\mapsto f(x)
\]
is an $\R$-point of the $\cin$-ring $\cin(M)$.

Conversely, given an $\R$-point $p:\cin(M) \to \R$ there is a point $x\in M$
so that $ev_x = p$.  This fact  is known as ``Milnor's exercise.''  It is a
theorem of Pursell \cite{Pursell}.  Thus for a manifold $M$ the set of
$\R$-points of $\cin(M)$  is (in bijection with) the set of ordinary
points of $M$.
\end{example}

\begin{remark}
  The kernel of an $\R$-point $p: \scC \to \R$ of a $\cin$-ring $\scC$
  is a maximal ideal in (the algebra underlying) $\scC$.  The converse
  is false:  there is a $\cin$-ring $\scC$ with a maximal ideal
  $\mathfrak{M}$ so that the field $\scC/\mathfrak{M}$ is not
  isomorphic to $\R$.   These kinds of maximal ideals do not seem to
  be useful in algebraic geometry over $\cin$-rings --- one uses
  $\R$-points exclusively.  
  \end{remark}

We now formally state the definition of a local $\cin$-ringed briefly
mentioned in the introduction.

\begin{definition}
A $\cin$-ring is {\sf local} if it has exactly one $\R$-point.
\end{definition}  

\begin{example} \label{ex:2.36}
Let $M$ be a manifold.  The stalk  $\cin_{M,x}$of the sheaf $\cin_M$ of smooth
functions on $M$ at a point $x$ consists of germs of functions at
$x$.   It is a $\cin$-ring \cite{MR} and it is a local $\cin$-ring:
the unique $\R$-point  $p:\cin_{M,x}\to \R$ is the evaluation at the point $x$.  See also Lemma~\ref{lem:2.58} below, which proves a more general result.
\end{example}  

\begin{remark}
There are $\cin$-rings with no $\R$-points.   The simplest such ring is the
zero ring (which corresponds to the product preserving functor $\Euc\to
\Set$ that sends every object of $\Euc$ to the same 1 point set $*$).

There are also {\em nonzero} $\cin$-rings with no $\R$-points.  To give an
example we need to recall that if $\scC$ is a $\cin$-ring and
$I\subset \scC$ is the ideal in the $\R$-algebra underlying $\scC$,
then the quotient algebra $\scC/I$ is a $\cin$-ring.  See \cite{MR}.
Granted this fact, let $\scC = \cin(\R^n)$, $n>0$ and $ \cin_c (\R^n)$
the subset of compactly supported functions.  Then $\cin_c (\R^n)$ is
an ideal in $\cin(\R^n)$ and the quotient ring
$\cin(\R^n)/\cin_c (\R^n)$ has no $\R$-points.  The latter fact essentially
follows from Pursell's theorem, which was mentioned above.
\end{remark}

\begin{definition} \label{def:lcrs}
A {\sf local $\cin$-ringed space}
($\LCRS$) is a pair $(M,\scA)$ where $M$ is a topological space and
$\scA$ is a sheaf of $\cin$-rings with the additional property that
the stalks of the sheaf are local $\cin$-rings.   We will refer to the
sheaf $\scA$ as the {\sf structure sheaf} of an $\LCRS$ $(M,\scA)$.
\end{definition}

\begin{example}
It follows from Example~\ref{ex:2.36} that a manifold $M$ together
with the sheaf $\cin_M$ of smooth functions is a local $\cin$-ringed
space.  
\end{example}

If manifolds were the only examples of local $\cin$-ringed spaces,
this paper would not have much of a point.   Fortunately there are two
broad classes of examples: affine $\cin$-schemes of Dubuc and
differential spaces in the sense of Sikorski.  We next recall what
they are.   We start with differential spaces.  However, before we do
that, we recall the definition of a point-determined $\cin$-ring.  We
will need this later when we discuss derivations.

\begin{definition} \label{def:pd}
A $\cin$-ring $\scC$ is {\sf point-determined} if $\R$-points separate
elements of $\scC$.  That is, if $a\in\scC$ and $a\not = 0$ then
there is an $\R$-point $p:\scC \to \R$ so that $p(a)\not = 0$.
\end{definition}

\begin{definition}[Differential space in the sense of Sikorski]\label{def:sikorski}
 A {\sf differential space} is a pair $(M, \scF)$ where $M$ is a
 topological space and $\scF$ is a (nonempty) set of real valued functions on $M$ so that
 \begin{enumerate}
   \item The topology on $M$ is the smallest topology making all
     functions in $\scF$ continuous;
 \item For any natural number $k$, for any $a_1,\ldots, a_k$ and any
   smooth function $f\in \cin(\R^k)$ the composite $f\circ
   (a_1,\ldots, a_k)$ is in $\scF$; \label{closure}\label{def:sikorksi:it2}
\item If $g: M\to \R$ is a function with the property that  for any $x\in M$ there is an
  open neighborhood $U$ of $x$ and $f\in \scF$ ($f$ depends on $x$)
  so that $g|_U = f|_U$ then $g\in \scF$. \label{def:sikorksi:it3}
\end{enumerate}
The  whole set $\scF$ is called a  {\sf differential structure} on $M$.
\end{definition}

\begin{remark}\mbox{}
\noindent
  \begin{itemize}
 \item
One thinks of the family of functions $\scF$ on a differential
space $(M, \scF)$ as the set of all abstract ``smooth'' functions on
$M$. 
\item Condition (\ref{def:sikorksi:it2}) of
  Definition~\ref{def:sikorski} says that the smooth structure $\scF$ is a
  $\cin$-ring: for any $n>0$,
$a_1, \ldots, a_n\in \scF$ and $f\in \cin(\R^n)$ the composite
$f\circ (a_1,\ldots, a_n):M\to \R$ is again in $\scF$ by
(\ref{def:sikorksi:it2}).  Thus setting 
\[
f_{\cin(M)} (a_1,\ldots, a_n) := f\circ (a_1,\ldots, a_n)
\]
for $f\in \cin(\R^n)$ and  $a_1, \ldots, a_n\in
\scF$) makes $\scF$ into a $\cin$-ring. 

\item Condition  (\ref{def:sikorksi:it3}) of
  Definition~\ref{def:sikorski} is a condition on the sheaf of
  functions on $M$ generated by $\scF$.  See Remark~\ref{rmrk:7.2}.
  It is often phrased as: ``if a function $g:M\to \R$ is locally smooth then it is smooth.''
\end{itemize}
\end{remark}

\begin{example} \label{ex:mfld-diffspace} Let $M$ be a smooth manifold.
  Then the pair $(M, \cin(M))$, where $\cin(M)$ is the set of
  $\cin$ functions, is a differential space.

  The only hard thing to check is that
  the topology on $M$ is the smallest among the topologies for which
  smooth functions are continuous.  But this follows from a theorem of
  Whitney: any closed subset of $M$ is the zero set of some smooth
  function (all our manifolds are Hausdorff and second
  countable).  See, for example, \cite[Theorem~2.29]{Lee}.  Hence if
  $U\subset M$ is open then there is a smooth function $f:M\to \R$ so
  that $U =\{x\in M \mid f(x) \not = 0\}$.
\end{example}

\begin{example}
  Let $M$ be a smooth manifold. Then the pair $(M, C^0(M))$, where
  $C^0(M)$ is the set of continuous functions, is a differential
  space.
\end{example}

\begin{example} \label{ex:ds_pd}
  If $(M, \scF)$ is a differential space then the $\cin$-ring $\scF$ is
  point-determined. This is because   a function on $M$ is non-zero only if there is a
  point at which it takes a non-zero value.
\end{example}

Every differential space is (secretly?) a local $\cin$-ringed space.
Here are the details.

\begin{definition}[The structure sheaf $\scF_M$ of a differential space
  $(M, \scF)$] \label{def:str_sheaf}
Let  $(M, \scF)$ be a differential space.  Define a presheaf $\cP$
of $\cin$-rings on $M$ as follows: for an open set $U\subset M$ let 
\[
  \cP(U): = \{ f|_U \mid f\in \scF\}
\]
and let the structure maps of $\cP$  be the restrictions of functions.
Define the {\sf structure sheaf} $\scF_M$ of  the differential space $(M, \scF)$ to be the
sheafification of $\cP$.
\end{definition}
The structure sheaf $\scF_M$ has a concrete description: 

\begin{lemma} \label{lem:7.1} Let  $(M, \scF)$ be a differential
  space,  $\cP$ the presheaf on $M$ defined by 
 \[
   \cP(U): = \{ f|_U \mid f\in \scF\}
 \]
 as in Definition~\ref{def:str_sheaf}. 
  For an open subset
$U\subseteq M$ let 
\[
\begin{split}  
  \scF_M(U) :=& \{f:U\to \R\mid \textrm{ for any }x\in U \textrm{ there is
    an open neighborhood } V \\
  &\textrm{ of }x \textrm{ in }U
  \textrm{ and } g\in \cin(M) \textrm{ with } g|_V = f|_V\}.
\end{split} 
\]  
Define the restriction maps $\rho^U_V: \scF_M(U) \to \scF_M(V)$ to be the restrictions of functions: $\rho^U_V (h):= h_V$.
 Then $\scF_M :=\{\scF_M(U)\}_{U\in \Open(M)}$ is a sheaf.  Moreover $\scF_M$ is a sheafification of the presheaf $\cP$. 
\end{lemma}

\begin{proof}[Proof of Lemma~\ref{lem:7.1}]
It is easy to see that the restriction maps $\rho^U_V: \scF_M(U) \to \scF_M(V)$ are well-defined. Consequently $\scF_M$ is a presheaf.  Moreover  $\cP(U) \subset
\scF_M (U)$ for all opens $U\in \Open(M)$.  Furthermore the presheaves $\scF_M$ and $\cP$ have the same
stalks.  So to prove that $\scF_M$ is a sheafification of $\cP$ it is enough to
check that $\scF_M$ is a sheaf.  This is not hard.
Furthermore  is not difficult to  check that $\scF_M$ is a sheaf of
$\cin$-rings.
\end{proof}

\begin{remark} \label{rmrk:7.2}
We now see what condition~(\ref{def:sikorksi:it3}) of
  Definition~\ref{def:sikorski} is saying.  It says that  the $\cin$-ring of functions
  $\scF$ on a differential space $(M, \scF)$ is the $\cin$-ring of
  global sections of the sheaf $\scF_M$ induced on $M$ by $\scF$
  and that  we don't get any more global sections.
\end{remark}

\begin{remark} \label{rmrk:2.57}
Note that for an open subset $U$ of $M$ a function $f$ is in
$\scF_M(U)$ if and only if there is an open cover $
\{U_\alpha\}_{\alpha\in A} $ of $U$ and a set $\{g_\alpha\}_{\alpha
  \in A}\subset \scF$ so that $f|_{U_\alpha} = g_\alpha|_{U_\alpha}$
for all $\alpha$.
\end{remark}  
\begin{lemma} \label{lem:2.58}
Let $(M, \scF)$ be a differential space and $\scF_M$ the induced sheaf
of $\cin$-rings.  Then the stalks of $\scF_M$ are local, hence the
pair $(M,
\scF_M)$ is a local $\cin$-ringed space.
\end{lemma}  
\begin{proof}
  For a point $x\in M$ the stalk $\cP_x$ is the set
germs of functions at $x$.  There is an evident $\R$-point
$ev_x:\cP_x\to \R$, $ev_x ([f]) = f(x)$.  If $f(x) \not =0$ then the
germ $[f]$ is invertible in the stalk $\cP_x$.  Hence $\ker ev_x$ is a unique maximal ideal
in $\cP_x$ and therefore $\cP_x$ is a local $\cin$-ring.
\end{proof}

\begin{remark} \label{rmrk:2.61}
Let $(M, \scF)$ be a differential space and $Y\subset M$ a subset.
Then there is a unique differential space structure $\scF_Y$ on $Y$: it
is the smallest differential structure on $Y$ containing the set
$\{f|_Y\mid f\in \scF\}$ of restrictions of function in $\scF$ to
$Y$.  See \cite{Sn}.

If $M$ is a manifold and $Y\subset M$ is closed then one can show that
$\scF_Y$ is
simply the set of restrictions $\scF|_Y$.  Therefore any closed subset
of a manifold $M$ gives rise to an easily understandable differential
space.  These differential spaces can be pretty wild.  For instance
the standard Cantor set $C$ in $\R$ with its induced differential
structure $\cin(C)\equiv\cin(\R)|_C$
is a differential space.
\end{remark}

Differential spaces form a category that we will denote by $\DiffSp$.
The morphisms of this category are smooth maps that are defined as
follows.

\begin{definition}  A  {\sf smooth map} from a
  differential space  $(M,\scF)$  to a differential space $(N,
  \scG)$ is a function 
  $\varphi:M\to N$ so that for
  any $f\in \scG$ the composite $f\circ \varphi$ is in $\scF$.
\end{definition}

We have observed above that to every differential space $(M,\scF)$ one
can associate a local $\cin$-ringed space $(M, \scF_M)$ (see
Lemma~\ref{lem:2.58}).  This map can be upgraded to a fully faithful
functor $F:\DiffSp \to \LCRS$.  Thus the category of differential
spaces is (isomorphic to) a full subcategory of the category $\LCRS$
of local $\cin$-ringed spaces.  See \cite{LdR} for a proof of this claim.

We next discuss $\cin$-schemes.  The idea of introducing an analogue
of algebraic geometry over
$\cin$-rings is due to Dubuc
\cite{Dubuc}.  It  has been developed further by Joyce \cite{Joy}.  Local
$\cin$-ringed spaces form a category.  Following Joyce \cite{Joy} we
denote this category by $\LCRS$.  For completeness we recall the
definition of morphisms in $\LCRS$.  Note first that any morphism of
local $\cin$-rings automatically preserves their $\R$-points.
(see \cite{MR} or \cite{Joy}).

\begin{definition}
A {\sf morphism} of local $\cin$-ringed spaces from $(X,\scA_X)$ to
$(Y,\scA_Y)$ is a pair $(f, f_\#)$ where $f: X\to Y$ is continuous and
$f_\#: \scA_Y\to f_* \scA_X$ is a map of sheaves of $\cin$-rings on the space $Y$.
\end{definition}

There is an evident global sections functor $\Gamma: \LCRS \to \op{\cring}$
from the category of local $\cin$-ringed spaces to the category opposite  of the category
$\cin$-rings.  It is given  by 
\[
\Gamma\left( (N, \scB)) \xrightarrow{(f,f_\#)} (M,\scA) \right) := 
\left(\scA(M)\xrightarrow{f_\#} (f_*\scB) (M) = \scB(f\inv (M)) = \scB (N)\right).
\]
Thanks to a
theorem of Dubuc \cite{Dubuc}, the global section functor $\Gamma$ has
a right adjoint
\[
  \Spec: \op{\cring} \to \LCRS ,
\]
the {\sf spectrum functor}.  One way to construct $\Spec$ is to proceed in
the same way as in algebraic geometry.  Namely, given a $\cin$-ring
$\scC$ one considers the set $\cX_\scC$ of all $\R$-points of $\scC$.
Then $\cX_\scC$ is given a topology and finally one constructs a sheaf
$\scO_\scC$ of local $\cin$-rings on the space $\cX_\scC$.  See
\cite{Joy}  for details.

\begin{definition} An affine {\sf $\cin$-scheme} is a local
  $\cin$-ringed space isomorphic to $\Spec(\scA)$ for some $\cin$-ring
  $\scA$.

The {\sf category of affine $\cin$-schemes} $\Aff$ is the essential
image of the functor $\Spec$.
\end{definition}

\begin{remark}
For any manifold $M$ the local $\cin$-ringed space $(M,\cin_M)$ is isomorphic to the affine scheme $\Spec(\cin(M))$ (see \cite{Dubuc} and/or \cite{Joy}).
In particular the real and complex projective spaces $\R P^n$ and $\C P^n$ are affine $\cin$-schemes.  Consequently 
 there is not the same motivation for considering projective $\cin$-schemes as there is for considering projective schemes.  
\end{remark}

\begin{remark}
  Unlike the spectrum functor in algebraic geometry
  $\Spec: \op{\cring} \to \Aff$ is not an equivalence of categories.
  For example $\Spec$ assigns the empty space to any $\cin$-ring with
 no $\R$-points, such as
  $\cin(\R^n)/\cin_c (\R^n)$ (here $\cin_c(\R^n)$ is the ideal of
  compactly supported functions) or the zero $\cin$-ring 0.
\end{remark}

\begin{remark}  The relation between $\cin$-schemes and differential
  spaces is a bit messy.   For a $\cin$-schemes to be a differential
  space its structure sheaf must be a sheaf of functions.  In
  particular it can have no nilpotents.  As as result
  many (most?) $\cin$-schemes are not differential spaces.

  On the other hand, given a differential space $(M,\scF)$ any point
  $x\in M$ gives rise to the $\R$-point $ev_x:\scF\to \R$, the
  evaluation at $x$.  However there are differential spaces where the
  set of $\R$-points is bigger than their  set of (ordinary) points,
  see \cite{CSt}.  Consequently there are differential spaces whose
  corresponding $\cin$-ringed spaces are not affine $\cin$-schemes.
\end{remark} 

We now turn our attention to $\Z$-graded presheaves and sheaves.
\begin{definition} \label{def:graded_ob}
  A {\sf $\Z$-graded  presheaf $\scP^\bullet$} of
$\scA$-modules over a local $\cin$-ringed space
$(M,\scA)$ is a sequence $\{\scP^n\}_{n\in \Z}$ of presheaves of
$\scA$-modules.  Similarly, a 
{\sf $\Z$-graded sheaf $\scS^\bullet $} of
$\scA$-modules over a local $\cin$-ringed 
$(M,\scA)$ is a sequence $\{\scS^n\}_{n\in \Z}$ of sheaves of
$\scA$-modules.
\end{definition}

\begin{remark} \label{rmrk:graded_sheaves}
Since the category of modules over a $\cin$-ring has coproducts so
does the category of presheaves of modules over a fixed local
$\cin$-ringed space $(M, \scA)$.  Moreover coproducts of presheaves 
are computed object-wise.  Consequently given a sequence 
$\{\scS^n\}_{n\in \Z}$ of presheaves of $\scA$-modules their coproduct
$ \oplus_{n\in \Z} \scS_n$ is defined by 
\[
\left( \oplus_{n\in \Z} \scS^n \right) (U)= \oplus_{n\in \Z} \scS^n(U),
\]
for all open sets $U\subset M$.  Here on the right the direct sum is
taken in the category of $\scA(U)$-modules.  Since the sheafification
functor $\sh:\Psh_M\to \Sh_M$ from presheaves to sheaves is left
adjoint to inclusion functor $i:\Sh_M\hookrightarrow \Psh_M$, the
functor $\sh$ preserves colimits and, in particular, coproducts.  Therefore if
$\{\scS^n\}_{n\in \Z}$ is a sequence of sheaves, then their coproduct (``direct sum'')
in $\Sh_M$ exists and equals $\sh (\oplus_{n\in \Z} \scS^n )$.  In
particular for an open set $U\subset M$, the $\scA(U)$-module
$\sh (\oplus_{n\in \Z} \scS^n ) (U)$ is, in general, different from
$\oplus_{n\in \Z} \scS^n (U)$. See Example~\ref{ex:pre-direct_sum}
below.  Of course $\sh (\oplus_{n\in \Z} \scS^n )$ (with the
appropriate structure maps) {\em is} a coproduct in $\Sh_M$, so one
can view $\Z$-graded objects in $\Sh_M$ as $\Z$-indexed coproducts.
One can even make sense of a map
$f: \sh (\oplus_{n\in \Z} \scS^n ) \to \sh (\oplus_{n\in \Z} \scS'^n
)$ as having a  degree $k$: we require that there are maps
$f^m:\scS^m\to {\scS'}^{m+k}$ of sheaves of modules making the diagrams
\[
\xy
(-15,10)*+{\sh (\oplus_{n\in \Z} \scS^n ) }="1";
(15,10)*+{\sh (\oplus_{n\in \Z} {\scS'}^n) }="2";
 (-15,-5)*+{\scS^m}="3";
(15,-5)*+{{\scS'}^{m+k}}="4";
{\ar@{->}^{f} "1";"2"};
{\ar@{->}^{\imath^m} "3";"1"};
{\ar@{->}_{{\imath'}^{m+k}} "4";"2"};
{\ar@{->}^{f^m} "3";"4"};
\endxy
\]
commute for all $m$ (here $\{\imath^m:\scS^m\to \sh (\oplus_{n\in \Z}
\scS^n ) \}_{m\in\Z}$  and $\{{\imath'}^m:{\scS'}^m\to \sh (\oplus_{n\in \Z}
{\scS'}^n ) \}_{m\in\Z}$ are the structure maps of the coproducts).
However it is not clear what the value of
the sheaf $\sh (\oplus_{n\in \Z} \scS^n )$ on some open subset
$U\subset M$ actually is in concrete situations.  This is why we use
Definition~\ref{def:graded_ob}.
\end{remark}

\begin{example}\label{ex:pre-direct_sum}  We  construct an example
  of a $\Z$-indexed collection of sheaves on a manifold so that their
  direct sum (as a presheaf) is not a sheaf.

  Let $B^n$ denote the open $n$-ball in $\R^n$ of radius 1 centered at
  0.  Let $M= \bigsqcup_{n\geq 0} B^n$, the coproduct of the balls $B^n$'s.  The
  space $M$ is a second countable Hausdorff manifold of unbounded
  dimension.  Let $\bOmega^k_{dR,M}$ denote the sheaf of ordinary
  differential $k$-forms on $M$ (with $\bOmega^k_{dR,M} \equiv 0$ for $k<0$).
Form the  direct sum (in the category of presheaves)
\[
\bOmega^\bullet_{dR,M}:= \oplus _{k\in \Z}\bOmega^k_{dR,M}.
\]
This presheaf {\it is not a sheaf}.  Here is one way to see it: for
each $n$ pick a volume form $\alpha_n\in \bOmega_{dR}^n (B^n)$.  The
collection of open balls $\{B^n\}_{n\geq 0}$ is an open cover of $M$
and all the elements of the cover are mutually disjoint.  If
$\oplus _{k\in \Z }\bOmega^k_{dR,M}$ were a sheaf, we would have an
element $\alpha\in \oplus _{k\geq 0 }\bOmega^k_{dR,M}$ so that
$\alpha|_{B^n} = \alpha_n$ for all $n$.  But any element of the direct
sum has a bounded degree, so no such $\alpha$ can exist.  Therefore
the direct sum $\oplus _{k\geq 0 }\bOmega^k_{dR,M}$ of sheaves is not
a sheaf.
\end{example}  
\begin{notation} \label{not:5.3}
Given a presheaf $\cP$ over a topological space $M$ we denote the
restriction maps from $\cP(U)$ to $\cP(V)$ for any two open sets
$V, U\in \Open(M)$ with $V\subset U$ by $\rho^U_V$ (with no reference
to the presheaf $\cP$ in the notation).  Thus
\[
\rho^U_V: \cP(U)\to \cP(V).
\]
Similarly if $(M, \scA)$ is a local $\cin$-ringed space and $\scM$ is
an $\scA$-module we write
\[
\rho^U_V: \scM(U)\to \scM(V)
\]
for the restrictions maps.   If we want to emphasize the dependence on
$\scM$ we write
\[
^\scM\!\rho^U_V: \cP(U)\to \cP(V)
\]
for the same restriction map.
\end{notation}  

\begin{notation}[$\Hom_{\scA}(\scN, \scM)$] \label{not:2.18}
Let $\scN, \scM$ be two $\scA$-modules over a (local) $\cin$-ringed space
$(M,\scA)$.    Recall that a map of $\scA$-modules $f:\scN \to \scM$
is a collection of maps $\{f_U:\scN(U)\to \scM(U)\}_{U\in\Open(M)}$
indexed by the poset $\Open(M)$ of opens subsets of $M$, where each $f_U$ is  a map of
$\scA(U)$-modules,  so that for any two pairs of open subsets
$V\subset U\subset M$ the diagram
\[
\xy
(-10,10)*+{\scN(U)}="1";
(10,10)*+{\scM(U)}="2";
(-10,-5)*+{\scN(V)}="3";
(10,-5)*+{\scM(V)}="4";
{\ar@{->}^{f_U} "1";"2"};
{\ar@{->}_{\rho_V^U} "1";"3"};
{\ar@{->}^{\rho_V^U} "2";"4"};
{\ar@{->}_{f_V} "3";"4"};
\endxy
\]
commutes (where as before  $\rho^U_V$ are  the restriction map, cf.\ Notation~\ref{not:5.3}).   We denote 
the set of all maps of $\scA$-modules from $\scN$ to $\scM$  by  $\Hom_{\scA}(\scN, \scM)$.
\end{notation}
\begin{remark} \label{rmrk:level1_3.2}
  The set
$\Hom_{\scA}(\scN, \scM)$ is a
$\scA(M)$-module  with the addition and action of $\scA(M)$
defined pointwise: if $a\in \scA(M)$, $T\in
\Hom_{\scA} (\scN, \scM)$ then $(a\cdot T)_U:= a(U) T_U$ for  all
open sets $U\subset M$.  Similarly for $S, T\in
\Hom_{\scA} (\scN, \scM)$, $(T+S)_U := T_U + S_U$ for all open
$U\subset M$.
\end{remark}

\begin{definition}[The $\scA(M)$-module $\cDer(\scA)$ of $\cin$-derivations of a $\cin$-ringed
  space $(M,\scA)$] \label{def:3.2level1}
  Let $(M,\scA)$ be a local $\cin$-ringed space.  A {\sf
    $\cin$-derivation of the sheaf $\scA$} is a map of presheaves of
 real vector spaces $v:\scA \to \scA$ so that for every open set $U\subset M$ the
  map $v_U:\scA(U)\to \scA(U)$ is a $\cin$-derivation of the
  $\cin$-ring $\scA(U)$.

We view the set $\cDer(\scA)$ of all
$\cin$-derivations of the sheaf $\scA$ is an $\scA(M)$-module: for
$a\in \scA(M)$ and $v\in \cDer(\scA)$, $(av)_U: = a|_U v_U$ for all
$v\in \Open(M)$.
\end{definition}

\begin{lemma} \label{lem:level03.6} Let $(M, \scA)$ be a local $\cin$-ringed space.
 The set $\cDer(\scA)$ of $\cin$-derivations of the structure sheaf $\scA$ is a real Lie
 algebra with the bracket defined by the commutator:
 \[
   [v,w] = v\circ w - w\circ v
 \]
for all $v, w\in \cDer(\scA)$.
\end{lemma}
\begin{proof}
View the sheaf $\scA$ of $\cin$-rings as a sheaf of vector spaces.
Then $\Hom(\scA, \scA)$ is a real  associative algebra under composition,
hence a Lie algebra with the bracket given by the commutator.

Any  $v\in \cDer(\scA) \subset \Hom(\scA, \scA)$ is a collection of
maps $v = \{v_U:\scA(U) \to \scA(U)\}_{U\in \Open(M)}$ compatible with
the restrictions $\rho^U_V: \scA(U)  \to \scA(V)$ with each
$v_U$  a $\cin$-derivation of the $\cin$-ring $\scA(U)$.  Given $v,
w\in \cDer(\scA)$ their commutator $[v, w]: = v\circ w - w\circ v$ is the collection of maps
$ \{([v,w])_U\} _{U\in \Open(M)}$ with $([v,w])_U = [v_U, w_U]$ for all
  $U\subset M$. By Lemma~\ref{lem:cder-lie} each commutator $[v_U,
  w_U]$ is in $\cDer(\scA(U))$.  Hence $\cDer(\scA)$ is a Lie
  subalgebra of $\Hom(\scA, \scA)$ and, in particular, a Lie algebra.
\end{proof}

\begin{remark}
In the case where the sheaf $\scA$ on a space $M$ is the structure
sheaf $\scF_M$ of a differential space $(M, \scF)$ (see
Definition~\ref{def:str_sheaf} above), the Lie algebra $\cin\Der
(\scA)$ is isomorphic to the Lie algebra of derivations of the
$\cin$-ring $\scF$.  See Theorem~\ref{thm:a1} in Appendix~\ref{app:A}.

\end{remark}  

\begin{definition} \label{def:12.16} A {\sf presheaf of commutative
    graded algebras (CGAs)} $ (\scM^\bullet, \wedge)$ over a
  $\cin$-ringed space $(M, \scA)$ is a sequence of $\scA$-modules
  $\{\scM^k\}_{k\in \Z}$ together with a sequence of maps of
  presheaves $\wedge: \scM^k\otimes \scM^\ell \to \scM ^{k+\ell} $ of
  $\scA$-modules so that $\wedge $ is associative and graded
  commutative. That is, 
 we require that the appropriate diagrams in the category
of presheaves of $\scA$-modules commute.

Equivalently a presheaf $ (\scM^\bullet, \wedge)$ of CGAs over $(M, \scA)$ is a functor
\[
\op {\Open(M)}\to \CGA
\]
That is, for all open sets $U\subset M$ the pair
$ ( \{\scM^k(U)\}_{k\in \Z}, \wedge_U) $ is CGA over the $\cin$-ring
$\scA(U)$ and the restriction maps
$r^U_V: ( \scM^\bullet(U), \wedge_U) \to ( \scM^\bullet(V), \wedge_V)
$ are maps of CGAs (cf. Definition~\ref{def:catCGA}).

\end{definition}

\begin{definition} \label{def:12.16-1} A {\sf sheaf of commutative
    graded algebras (CGAs)} $ (\scM^\bullet, \wedge)$ over a
  $\cin$-ringed space $(M, \scA)$ is a presheaf of CGAs $
  (\scM^\bullet, \wedge)$  with the property that $\scM^k$ is a sheaf
  for each $k\in \Z$.
\end{definition}

\begin{definition} \label{def:gr-der_presh}
A {\sf graded derivation $X$ of degree $k$} of a presheaf $(\scM^\bullet
=\{\scM^k\}_{k\in \Z},
\wedge)$ of CGAs over a
  $\cin$-ringed space $(M, \scA)$  is a collection of maps of presheaves $\{X^\ell:
\scM^\ell \to \scM^{\ell +k}\}_{\ell \in \Z}$ of $\scA$-modules so that
\[
X_U: \scM^\bullet(U) \to \scM^{\bullet +k}(U)
\]
is a degree $k$ derivation of the CGA $(\scM^\bullet (U),
\wedge)$ of $\scA(U)$-modules
for each open
set $U\in \Open(M)$.     Equivalently $X =\{X^\ell\}_{\ell\in \Z}$ makes the appropriate diagrams
in the category of presheaves of $\scA$-modules commute.
\end{definition}

\begin{notation}[$\Der^k(\scM^\bullet)$]\quad Let $(\scM^\bullet,\wedge)$ be a presheaf of CGAs
over a $\cin$-ringed space $(M, \scA)$.  Denote the set of all graded
derivations of degree $k$ of $(\scM^\bullet, \wedge)$ by
$\Der^k(\scM^\bullet)$.
\end{notation}

\begin{remark}
The set $\Der^k(\scM^\bullet)$ is easily seen
to be an $\scA(M)$-module (cf.\ Remark~\ref{rmrk:level1_3.2}). The two operations are defined by:
\begin{eqnarray*}
\{X^\ell\}_{\ell\in \Z} + \{Y^\ell\}_{\ell\in \Z}: =
  \{X^\ell+Y^\ell\}_{\ell\in \Z},\\
a \cdot \{X^\ell\}_{\ell\in \Z}  := \{a\cdot X^\ell\}_{\ell\in \Z}
\end{eqnarray*}
for all $X=\{X^\ell\}_{\ell\in \Z}, Y=\{Y^\ell\}_{\ell\in \Z}    \in
\Der^k (\scM^\bullet)$ and all $a\in \scA(M)$.
\end{remark}

\begin{definition}
A {\sf presheaf of commutative differential graded algebras (CDGAs)} $
(\scM^\bullet, \wedge, d)$ over a $\cin$-ringed space $(M,\scA)$ is a
presheaf $ (\scM^\bullet, \wedge)$ of CGAs over $(M,\scA)$ together
with a sequence of maps $\{d:\scM^k \to \scM^{k+1}
\}_{k\in \Z}$ of presheaves of real vector spaces so that for each
open set $U\subset M$ the triple $(\scM^\bullet(U), \wedge_U, d_U)$ is
a CDGA over the $\cin$-ring $\scA(U)$.

Equivalently a presheaf CDGAs over a $\cin$-ringed space $(M, \scA)$
is a functor from the category $\op{\Open(M)}$ of open subsets of $M$
to the category $\CDGA$ of commutative differential graded algebras
(cf.\ Definition~\ref{def:catCDGA} ).
\end{definition}

\begin{definition}
A {\sf sheaf of CDGAs} is a presheaf of CDGAs
$(\scS^\bullet, \wedge, d)$ so that $\scS^k$ is a sheaf for all
$k\in \Z$.
\end{definition}

\begin{proposition} \label{prop:3.15level1} Let
  $(\scM^\bullet, \wedge)$ be a presheaf of CGAs over a local
  $\cin$-ringed space $(M,\scA)$. %
    Then
\begin{enumerate}
\item \label{3.16.i} For any two graded derivations $X\in \Der^{|X|}(\scM^\bullet)$,
  $Y\in \Der^{|Y|}(\scM^\bullet)$ the graded commutator
\[  
[X,Y]:= X\circ Y - (-1)^{|X||Y|} Y\circ X
\]
is a graded derivation of $(\scM^\bullet, \wedge)$ of degree $|X|+|Y|$.
\item $\Der^\bullet (\scM^\bullet ) = \{\Der^k(\scM^\bullet)\}_{k\in \Z}$ is a graded (real)
  Lie algebra with the bracket defined by the graded commutator.
\end{enumerate}
\end{proposition}

\begin{proof}
Given two derivations $X\in \Der^k (\scM^\bullet)$ and $Y\in \Der^\ell
(\scM^\bullet)$ their graded commutator $[X,Y]$ makes sense as a
collection of maps of presheaves
\[
\left\{([X,Y])^n = X^{\ell + n} \circ Y^n - (-1)^{k\ell}Y ^{k+n} \circ X^n:
\scM^n \to \scM^{n+k+\ell}\right \}_{n\in \Z}.
\]
One checks that the graded commutator $[X,Y]$ is a derivation of
degree $|X|+|Y| = k+\ell$ and that the pair 
\[
(\Der^\bullet (\scM^\bullet), [\cdot, \cdot] ) = \left(\{\Der^k(\scM^\bullet)\}_{k\in \Z},
[\cdot, \cdot] \right)
\]
is a graded Lie algebra.
\end{proof}

\subsection*{Internal hom and related constructions}

The following definition is an ``obvious" analogue of the corresponding
well-known definition in algebraic geometry.

\begin{definition} \label{def:4.1level2}
Let $(M,\scA)$ be a $\cin$-ringed space and $\scM$, $\scN$ two
presheaves of $\scA$-modules. %
We then can form a presheaf $\cHom(\scM, \scN)$ of
$\scA$-modules. We call this presheaf the {\sf internal hom} of
$\scM$ and $\scN$ or the {\sf (pre)sheaf-hom}.  Namely set
\[
\cHom(\scM, \scN) (U): = \Hom(\scM|_U, \scN|_U)
\]
for all open subsets $U\subset M$. Define the restriction maps
$\rho^U_V:  \Hom(\scM|_U, \scN|_U) \to  \Hom(\scM|_V, \scN|_V)$ by
\[
\rho^U_V \left(\{f_W\}_{W\in\Open(U)} \right) = \{f_W\}_{W\in\Open(V)}.
\]
(cf.\ Notation~\ref{not:2.18}).
\end{definition}

\begin{remark} \label{rmrk:int-hom-sheaf}
 The internal hom
  $\cHom(\scM, \scN) $ is a sheaf if $\scN$ and $\scM$ are sheaves.
  This is well-known.
\end{remark}

\begin{remark} \label{rmrk:int-hom-sheaf2}
In fact the internal hom $\cHom(\scM, \scN) $ is a sheaf if just the
target presheaf $\scN$ is a sheaf.  This is because for any open set
$U\in \Open(M)$ any map of presheaves $\varphi: \scM|_U\to \scN|_U$
canonically factors through the sheafification $\eta_U:\scM\to
\sh(\scM|_U) \simeq \sh(\scM)|_U$, which, in turn gives us a bijection
$\eta_U^*: \Hom(\sh(\scM)|_U, \scN|_U)\to \Hom(\scM|_U, \scN_U)$.
These bijections assemble into an isomorphism of presheaves
\[
  \eta^*: \cHom(\sh(\scM), \scN) \to \cHom(\scM, \scN) .
\]
Therefore,
  since  $\cHom(\sh(\scM), \scN)$ is a sheaf, so is $\cHom(\scM, \scN)$.
\end{remark}

\begin{definition}[The $\scA$-module of presheaves $\cin\Dert(\scA)$ 
  of $\cin$-derivations of a $\cin$-ringed
  space $(M,\scA)$] \label{def:3.2level2}
  Let $(M,\scA)$ be a local $\cin$-ringed space.   For each open set
  $U\in \Open(M)$ we define
  \[
\cin\Dert(\scA) (U):= \cDer(\scA|_U).
\]
For any pair of open sets $U,V\in\Open(M)$ with $V\subset U$ the
restriction maps $\rho^U_V : \cin\Dert(\scA) (U)\to \cin\Dert(\scA) (V)$
are given by
\[
\rho^U_V \left( \{v_W\}_{W\in \Open(U)}\right) = \{v_W\}_{W\in \Open(V)}.
\]
Consequently $\cin\Dert(\scA) := \{\cin\Dert(\scA) (U)\}_{U\in
  \Open(M)}$ is a presheaf of $\scA$-modules.  It's a subpresheaf of
$\cHom(\scA, \scA)$.
\end{definition}

\subsection*{The presheaf 
  of $\cin$-algebraic de Rham (a.k.a. differential) forms on a $\LCRS$
  $(M,\scA)$}\mbox{}\\[4pt]
We recall  from \cite{LdR} the construction of the presheaf of $\cin$-algebraic
differential forms $(\Lambda^\bullet \Omega^1_\scA, \wedge, d)$
over  a local
$\cin$-ringed space $(M,\scA)$.   It proceeds as follows: for each open set $U\in \Open(M)$ we have
a $\cin$-ring $\scA(U)$ hence the $\cin$-algebraic de Rham complex
$(\Lambda^\bullet \Omega^1_{\scA(U)}, \wedge, d)$ (see
Definition~\ref{def:dR_for_ring}).  We need to check that the
assignment $U\mapsto (\Lambda^\bullet \Omega^1_{\scA(U)}, \wedge, d)$
extends to a functor from $\op{(\Open(M))}$ to the category $\CDGA$ of
CDGAs.   We already know that $U\mapsto (\Lambda^\bullet
\Omega^1_{\scA(U)}, \wedge)$ is a presheaf of CGAs.  The issue is the
naturality of the differential $d$, which follows from the following proposition.

\begin{proposition} \label{prop:6.5}
A map $f:\scC\to \scB$ of $\cin$-rings induces a unique map
$\Lambda^\bullet(f): \Lambda^\bullet\Omega^1_\scC \to \Lambda^\bullet\Omega^1_\scB$ of
commutative differential graded algebras.  Explicitly
\[
  \Lambda^\bullet(f) (a_0 d_\scC a_1\wedge \cdots \wedge d_\scC a_k)
  = f(a_0) d_\scB f(a_1)\wedge \cdots \wedge d_\scB f(a_k)
\]  
for all $k>0$ and all $a_0,\ldots, a_k\in \scC$.
\end{proposition}
\begin{proof}
See \cite[Proposition~7.17]{LdR}.
\end{proof}

\begin{notation}[The presheaf $(\Lambda^\bullet \Omega^1_\scA, \wedge,
  d)$ of $\cin$-algebraic differential forms over a $\LCRS$ $(M,\scA)$]
Proposition~\ref{prop:6.5} implies that for a local $\cin$-ringed spaces $(M, \scA)$ the
assignment
\[
\Open(M)\ni U\to (\Lambda^\bullet \Omega^1_{\scA(U)}, \wedge, d)
\]  
is a presheaf of commutative differential graded algebras.  We denote
this presheaf by $(\Lambda^\bullet \Omega^1_\scA, \wedge, d)$ and
call it the {\sf presheaf of $\cin$-algebraic differential forms}
(and  the {\sf presheaf of $\cin$-algebraic de Rham forms}; we
use the two terms interchangeably).
\end{notation}

\begin{remark}
Note that given a presheaf $(\Lambda^\bullet \Omega^1_\scA, \wedge,
d)$ of $\cin$-algebraic differential forms we can also view $d$ as a
graded derivation of degree 1 of the presheaf of CGA $(\Lambda^\bullet
\Omega^1_\scA, \wedge)$, cf. Remark~\ref{rmrk:values}.
\end{remark}

\begin{remark} \label{rmrk:2.82june}
An attentive reader may have noticed by this point that we use the
symbol $d$ to denote several different (but related) maps:
\begin{itemize}
\item[(i)] $d= d_\scC:\scC\to \Omega^1_\scC$, the universal derivation of a
  $\cin$-ring $\scC$.

\item[(ii)] $d= d_\scC:\Lambda^\bullet \Omega^1_\scC \to
  \Lambda^{\bullet} \Omega^1_\scC $, the extension of $d$ of (i)
  to a graded derivation of the CGA $ \Lambda^{\bullet +1}
  \Omega^1_\scC $ of $\cin$-algebraic differential forms of a
  $\cin$-ring $\scC$

\item[(iii)] $d= d_\scA:  \Lambda^\bullet \Omega^1_\scA \to
\Lambda^{\bullet +1} \Omega^1_\scA$, a map of graded presheaves of
modules over a local $\cin$-ringed space $(M,\scA)$ (and a graded
derivation of the  presheaf of CGAs $(\Lambda^\bullet \Omega^1_\scA,
\wedge)$).  Note that
\[
d = \{d_W: \Lambda^\bullet \Omega^1_\scA(W) \to \Lambda^{\bullet +1}
\Omega^1_\scA(W)\}_{W\in \Open(M)}
\]
where each $W$ component $d_W\in \Der^1 (\Lambda^\bullet
\Omega^1_\scA(W)) \equiv \Der^1 (\Lambda^\bullet
\Omega^1_{\scA(W)}) $ is really $d_{\scA(W)}$ of (ii).
\end{itemize}

\end{remark}

\begin{lemma} \label{lem:6.5}
The sheaf of $\cin$-derivations  $\cin \Dert (\scA)$ of a locally $\cin$-ringed space $(M,\scA)$ is a sheaf of Lie algebras.
\end{lemma}

\begin{proof}
For $U\in \Open(M)$ the $\scA(U)$-module $\cin\Dert(\scA)(U)$ is
$\cin\Der(\scA|_U)$, by definition.  The Lie bracket $[\cdot,
\cdot]_U$ on the vector
space $\cin\Der(\scA|_U)$ is the commutator :
\[
[v,w]_U := v\circ w - w\circ v
\]
for all $v,w\in \cin\Der(\scA|_U)$.
Since the restrictions maps $\rho^U_V: \cin\Der(\scA|_U) \to
\cin\Der(\scA|_V)$ preserve composition,
\[
\rho^U_V ([v,w]_U)  := \rho^U_V (v\circ w - w\circ v) = [\rho^U_V(v),
\rho^U_V(w)]_V
\]
for all  $v,w\in \cin\Der(\scA|_U)$.   Thus $\cin \Dert (\scA)$ is a sheaf of Lie algebras.
\end{proof}

\begin{definition}[The presheaf $\Dert^k (\scM^\bullet)$ of graded
  derivations of degree $k$ of a presheaf $(\scM^\bullet, \wedge)$ of
  commutative graded algebras.]
Let $(M, \scA)$ be a local $\cin$-ringed space and $(\scM^\bullet,
\wedge)$ a presheaf of CGAs over $(M,\scA)$.  We define the presheaf 
$\Dert^k (\scM^\bullet)$ of graded derivations of degree $k$ of
$(\scM^\bullet, \wedge)$ by setting
\[
\Dert^k (\scM^\bullet) (U) := \Der^k (\scM^\bullet|_U) 
\]
for all open subsets $U\subset M$.   The restriction maps $\rho^U_V:
\Dert^k (\scM^\bullet) (U) \to \Dert^k (\scM^\bullet) (V)$ are defined
by
\[
\rho^U_V \left(\{X_W\}_{W\in\Open(U)} \right) = \{X_W\}_{W\in\Open(V)}.
\]
(compare with Definition~\ref{def:4.1level2}).
\end{definition}

\begin{lemma}  Let $(\scM^\bullet, \wedge)$ be a presheaf of CGAs
  over a $\cin$-ringed spaces $(M,\scA)$ and 
  $\Dert^\bullet (\scM^\bullet)= \{\Dert^k (\scM^\bullet)\}_{k\in
    \Z}$ the presheaf of graded derivations. Then  $\Dert^\bullet
  (\scM^\bullet) $ is a presheaf of real graded Lie algebras.
\end{lemma}

\begin{proof}
We need to check that for   any pair of open sets $V,U\in \Open(M)$ with $V\subset U$, the
restriction map 
\[
\rho^U_V :   \Dert^\bullet (\scM^\bullet) (U) \to \Dert^\bullet (\scM^\bullet) (V)
\]
is a map of Lie algebras.
By definition of the presheaf $\Dert^\bullet (\scM^\bullet)$, its
value $\Dert^\bullet (\scM^\bullet)(U)$ on $U\in \Open(M)$ is $\Der^\bullet (\scM^\bullet|_U)$, which is
a  graded Lie algebra with the bracket given by the graded commutator:
for all $X,Y\in \Der^\bullet (\scM^\bullet|_U)$
\[
[X,Y]:= X\circ Y -  (-1) ^{|X||Y|} Y\circ X.
\]  
Note that $X = \{X_W\in \Der^\bullet (\scM^\bullet|_U) (W)\}_{W\in
  \Open(M)}$, $Y = \{Y_W \in \Der^\bullet (\scM^\bullet|_U) (W)\}_{W\in
  \Open(M)}$ and $X\circ Y = \{X_W\circ Y_W\}_{W\in \Open(U)}$.
Consequently
\[
  [X,Y] = \{[X_W, Y_W]\}_{W\in \Open(M)}.
\]
Now for any pair of open sets $V,U\in \Open(M)$ with $V\subset U$, the
restriction map 
\[
\rho^U_V :   \Dert^\bullet (\scM^\bullet) (U) \to \Dert^\bullet (\scM^\bullet) (V)
\]
is given by
\[
\rho^U_V (\{Z_W\}_{W\in \Open(U)} ) = \{Z_W\}_{W\in \Open(V)} .
\]  
Therefore 
\[
\rho^U_V ([X,Y]) = \{([X, Y])_W\}_{W\in \Open(V)} = \{[X_W,
Y_W]\}_{W\in \Open(V)} = [\rho^U_V(X),  \rho^U_V (Y)]
\]
and we are done.
\end{proof}

\begin{proposition} \label{prop:ad(x)} Let $(\scM^\bullet, \wedge)$ be
  a presheaf of CGAs over a $\cin$-ringed spaces $(M,\scA)$. Then any
  derivation
  $X\in \Der^\bullet (\scM^\bullet) \equiv
  \Dert^\bullet(\scM^\bullet)(M)$ of degree $k$ gives rise to a map of
  presheaves of real vector spaces
\[
 {ad}(X): \Dert^\bullet(\scM^\bullet)\to \Dert^{\bullet+k} (\scM^\bullet)
\]
with components 
${ad}(X)_U : \Dert^\bullet(\scM^\bullet)(U) \to
\Dert^{\bullet+k} (\scM^\bullet)(U) $ given by
\[
{ad}(X)_U \,(Y) := \{[X_W, Y_W]\}_{W\in \Open(U)}
\]

for all $Y\in \Dert^\bullet(\scM^\bullet)(U)$.
\end{proposition}   

\begin{proof}
We need to check that for   any pair of open sets $V,U\in \Open(M)$
with $V\subset U$
\[
\rho^U_V \circ {ad}(X)_U = {ad}(X)_V \circ \rho^U_V.
\]
For any $Y\in \Dert^\bullet(\scM^\bullet)(U)$
\[
\rho^U_V ({ad}(X)_U Y) = \rho^U_V (\{[X_W,  Y_W]\}_{W\in
  \Open(U)}) = \{[X_W,  Y_W]\}_{W\in\Open(V )}=  {ad}(X)_V
( \rho^U_V (Y)).
\]  
\end{proof}

\begin{corollary} \label{cor:6.11} For any presheaf $(\scM^\bullet, \wedge, d)$  of CDGAs over a
$\cin$-ringed spaces $(M,\scA)$ there is a map of presheaves
\[
  {ad}(d): \Dert^\bullet(\scM^\bullet)\to \Dert^{\bullet+1} (\scM^\bullet)  
\]
with ${ad}(d)_U (Y) = [d|_U, Y]$ for all $U\in \Open(M)$
and all $Y\in \Der^\bullet (\scM^\bullet |_M)$.
\end{corollary}  

\begin{remark} \label{rmrk:ad(d)}
In particular, since the presheaf $(\Lambda^\bullet\Omega^1_\scA,
\wedge, d)$ of
$\cin$-algebraic differential forms on a local $\cin$-ringed space
$(M,\scA)$ is a presheaf of CDGAs, Corollary~\ref{cor:6.11} implies
that we have a map of presheaves
\[
  {ad}(d): \Dert^\bullet(\Lambda^\bullet\Omega^1_\scA)\to
  \Dert^{\bullet+1} (\Lambda^\bullet\Omega^1_\scA) .
\]

\end{remark}

\begin{lemma} \label{lem:2.33june}
Let $(\scM^\bullet, \wedge, d)$ be a presheaf  of CDGAs over a
$\cin$-ringed spaces $(M,\scA)$ and $ {ad}(d):
\Dert^\bullet(\scM^\bullet)\to \Dert^{\bullet+1} (\scM^\bullet) $ the
corresponding map of presheaves of graded vector spaces (cf.\
Corollary~\ref{cor:6.11}).
Then \[
ad(d) \circ ad(d) =0.
\]
\end{lemma}

\begin{proof}
This is a computation.  We only need to use the facts that $d$ has degree 1 and that $d\circ
d =0$.  Here are the details. For any open set $U\in\Open(M)$ for any $k$ and any
derivation $Y\in \Der^k(\scM|_U)$
\begin{align*}
&(ad(d)_U \circ ad(d)_U ) \, Y = [d|_U, [d|_U, Y]] = d|_U\circ  [d|_U,
                                Y] -(-1)^{k+1}  [d|_U, Y] \circ d|_U\\
  &= d|_U\circ (d|_U\circ Y - (-1)^k Y\circ d|_U) +(-1)^k (d|_U\circ Y
    - (-1)^k Y\circ d|_U) \circ d|_U\\
  &= -(-1)^k d|_U\circ Y\circ d|_U + (-1)^k d|_U\circ Y\circ d|_U =0.
\end{align*}  
\end{proof}  

\begin{remark} \label{rmrk:2.90} \label{rmrk:ad(d)W}
For any $\cin$-ring $\scC$, the exterior derivative $d = d_\scC\in \Der^1
(\Lambda^\bullet \Omega^1_\scC)$ gives rise to
$ad(d_\scC): \Der^\bullet
(\Lambda^\bullet \Omega^1_\scC) \to (\Lambda^{\bullet
  +1}\Omega^1_\scC)$ just as in the case of presheaves (cf.\
Corollary~\ref{cor:6.11}).  Explicitly, for any $k$ and any  $Y\in \Der^k
(\Lambda^\bullet \Omega^1_\scC) $
\begin{equation} \label{eq:2.90.1}
ad(d_\scC)Y\equiv ad(d)Y = [d, Y] \equiv d\circ Y -(-1)^k Y\circ d.
\end{equation}
Consequently Proposition~\ref{prop:ad(x)} says that for any $\LCRS$
$(M, \scA)$, for any open set $U\subset M$ the $U$-component
\[
ad(d)_U: \Dert^\bullet(\Lambda ^\bullet\Omega^1_\scA )(U) \equiv \Dert^\bullet(\Lambda ^\bullet\Omega^1_\scA|_U)\to
\Dert^{\bullet+1} (\Lambda ^\bullet\Omega^1_\scA)(U) \equiv
\Dert^{\bullet+1} (\Lambda ^\bullet\Omega^1_\scA |_U)
\]  
is
\[
ad(d)_U =\{ ad(d_{\scA(W)})\}_{W\in \Open(U)},
\]  
where each $ad(d_{\scA(W)})$ is defined by \eqref{eq:2.90.1}.
\end{remark}

\section{Cartan calculus for $\cin$-rings} \label{sec:ring}

The goal of this section is to  prove
Theorem~\ref{prop:cartan_calc_level0}.   At the same time we do some
preparatory work for the proof of Theorem~\ref{main_thm}.  We start by
defining semi-conjugacy (i.e., relatedness) in the context of $\cin$-rings.

\begin{definition}
Let $f:\scC\to \scB$ be a map of $\cin$-rings.  A derivation $v\in
\cin\Der(\scC)$ is {\sf $f$-related} to a derivation $w\in
\cin\Der(\scB)$ if and only if
\[
  f\circ v = w\circ f.
\]  
\end{definition}  

The following definition will be useful in the construction of 
contraction and Lie derivative maps in the next section.
\begin{definition} \label{def:nat}
An $\R$-linear map $h_\scC: \cin\Der(\scC) \to \Der^k (\Lambda^\bullet
\Omega_\scC^1)$ is {\sf natural} in the $\cin$-ring $\scC$ if for any
map of $\cin$-rings $f:\scC\to \scB$ and for any
two derivations $v\in \cin\Der(\scC)$, $w\in \cin\Der(\scB)$ with the
property that $f\circ v = w\circ f$ (i.e., $v$ and $w$ are $f$-{ related}) we have
\[
\Lambda^\bullet (f) \circ h_\scC (v) = h_\scB (w) \circ\Lambda^\bullet (f).
\]  
\end{definition}

\begin{lemma} \label{lem:contr-for-rings}
For any $\cin$-ring $\scC$ there exists a unique $\scC$-linear map
\[
\iota_\scC:\cin\Der(\scC)\to \Der^{-1}(\Lambda^\bullet \Omega^1_\scC) ,
\]
the {\sf contraction},
so that its degree 1 component $\iota_\scC^{(1)}: \cin\Der(\scC)\to
\Hom(\Omega^1_\scC, \scC)$ is the inverse of $d^*: \Hom(\Omega^1_\scC,
\scC) \to \cin\Der(\scC)$ , $d^*\varphi:= \varphi \circ d$.

Moreover the map $\iota_\scC$ is natural in $\scC$ (Definition~\ref{def:nat}).
\end{lemma}

\begin{proof}
For any $\scC$-module $\scM$ and any $\scC$-linear map
$\varphi:\scM\to \scC$ and any $k>1$ consider the map 
\[
  \tilde{\varphi}^{(k)}:\overbrace{\scM\times \cdots \times \scM}^{k} \to
  \Lambda^{k-1} \scM\]
given by 
\[
\tilde{\varphi}^{(k)}(\alpha_1,\ldots,\alpha_k)  = \sum_s
(-1)^{s+1} \varphi(\alpha_s) \alpha_1\wedge \ldots\wedge
\widehat{\alpha_s} \wedge \ldots\wedge \alpha_k
\]
for all $\alpha_1,\ldots, \alpha_k\in \scM$.

Since $\tilde{\varphi}^{(k)}$ is $\scC$-linear in each slot and alternating, it descends to a
$\scC$-linear map 
\[
{\varphi}^{(k)}: \Lambda^k \scM \to \Lambda^{k-1} \scM
\]
with the property that
\[ 
  {\varphi}^{(k)}(\alpha_1\wedge \ldots \wedge \alpha_k) =
  \sum_i
(-1)^{i+1} \varphi(\alpha_i) \alpha_1\wedge \ldots\wedge
\widehat{\alpha_i} \wedge \ldots\wedge \alpha_k
\] 
 Moreover, 
for all $k,\ell >0$, and all
$\alpha_1,\ldots, \alpha_k, \alpha_{k+1}, \ldots, \alpha_{k+\ell}$
\[
\begin{split}  
{\varphi}^{(k+\ell)}(\alpha_1\wedge \ldots \wedge
\alpha_{k+\ell}) = \left(\varphi^{(k)}(\varphi) (\alpha_1\wedge \ldots
  \wedge \alpha_k) \right) \wedge ( \alpha_{k+1}\wedge \ldots\wedge
\alpha_{k+\ell} )\\ + (-1) ^{k(-1)} (\alpha_1\wedge \ldots
  \wedge \alpha_k )\wedge \left(\varphi^{(\ell)} (\alpha_{k+1}\wedge \ldots
    \wedge \alpha_{k+\ell}) \right) .
\end{split}  
\]
We now set $\varphi^{(1)} := \varphi$, $\Lambda^k \scM =0$ for $k<0$
and $\varphi^{(k)} = 0$ for $k<0$.
Then  $\varphi^\bullet:= \{ \varphi^{(k)}\} _{k\in \Z}$ is a degree -1
derivation of the CGA $\Lambda ^\bullet \scM$. 

If $\tau^\bullet =\{\tau^{(k)} \}_{k\in \Z}$ is another degree -1
derivation of $\Lambda^\bullet \scM$ with $\tau^{(1)} = \varphi$, 
then for
any $k>0$, any
$\alpha_1,\ldots,\alpha_k\in \scM$
\[
  \tau^{(k)}(\alpha_1\wedge \ldots \wedge \alpha_k) =
  \sum_i
(-1)^{i+1} \varphi(\alpha_i) \alpha_1\wedge \ldots\wedge
\widehat{\alpha_i} \wedge \ldots\wedge \alpha_k 
=  \varphi^{(k)}(\alpha_1\wedge \ldots \wedge \alpha_k).
\]
Hence the derivation $\varphi^\bullet$ is uniquely determined by its
degree 1 component $\varphi^{(1)}$.

We are now in position to construct the desired contraction
$\iota_\scC:\cin\Der(\scC)\to \Der^{-1}(\Lambda^\bullet
\Omega^1_\scC)$.   By the universal property of the derivation
$d_\scC:\scC\to \Omega^1_\scC$,
given a derivation $v\in \cin\Der(\scC)$ there exists a unique
$\cin$-linear map $\imath(v):\Omega^1_\scC\to \scC$ so that $\imath(v)
\circ d_\scC = v$.  Therefore there exists a unique degree -1
derivation $\imath(v)^\bullet $ of $\Lambda^\bullet \Omega^1_\scC$ with
$\imath (v)^{(1)} = \imath(v)$.  We define the desired map
$\iota_\scC:\cin\Der(\scC)\to \Der^{-1} (\Lambda^\bullet \Omega^1_\scA)$
by
\[
\iota_\scC(v): = \imath (v)^\bullet = \{ \imath(v)^{(k)} \}_{k\in \Z}
\]  
for all $v\in \cin\Der(\scC)$.  It is easy to see that $\iota_\scC$ is
$\scC$-linear.\\[4pt]
Next suppose $f:\scC\to \scB$ is a map of $\cin$-rings, and that the
derivations $v\in \cin\Der(\scC)$, $w\in \cin\Der(\scB)$ are
$f$-related.  Since the maps $\iota_\scC (v) $ and $\iota_\scB (w)$ are
uniquely determined by their degree 1 components (which are
$\imath(v)$ and $\imath(w)$, respectively), it is enough to
prove that the diagram
\begin{equation}\label{eq:3.3.1june}
\xy
(-15,10)*+{\Omega^1_\scC}="1";
(15,10)*+{\scC}="2";
(-15,-5)*+{\Omega^1_\scB }="3";
(15,-5)*+{\scB}="4";
{\ar@{->}^{\imath(v)} "1";"2"};
{\ar@{->}_{\Lambda^1 f} "1";"3"};
{\ar@{->}^{f} "2";"4"};
{\ar@{->}_{\imath(w)} "3";"4"};
\endxy
\end{equation}
commutes.  Since $f$ is a $\cin$-ring homomorphism, $w\circ f = f\circ
v:\scC\to \scB$ is a $\cin$-ring derivation of $\scC$ with values in
$\scB$ (thought of as a $\scC$-module by way of $f$).  By the universal
property of $d_\scC:\scC \to \Omega^1_\scC$ there exists a unique
$\scC$-linear map $\psi:\Omega^1_\scC\to \scB$  so that
\begin{equation} \label{eq:3.3.2june}
f\circ v = \psi \circ d_\scC = w \circ f.
\end{equation}

By definition of the contraction $\imath(v)$ we have
\[
f\circ v =
f\circ \imath(v)\circ d_\scC.
\]
By definitions of $\imath(w)$ and of the map $\Lambda^1f$
\[
w\circ f  =(\imath(w) \circ d_\scB) \circ f = \imath(w) \circ
\Lambda^1 f \circ d_\scC.
\]  
Therefore by the uniqueness of the map $\psi$ satisfying
\eqref{eq:3.3.2june},
\[
f \circ \imath(v) = \imath(w) \circ f,
 \] 
i.e., \eqref{eq:3.3.1june} commutes.
\end{proof}

\begin{definition}[Lie derivative map] \label{def:3.4june}
  Let $\scC$ be a $\cin$-ring.  We define the {\sf Lie derivative} map  
\[
\cL_\scC: \cDer(\scC) \to \Der^{0} (\Lambda^\bullet\Omega^1_\scC),
\]
by 
\[
\cL_\scC(v) :=[d, \iota_\scC(v)] 
\]  
for all $v\in \cDer(\scC)$, where on the right $[d, \iota_\scC(v)] $
is the graded commutator of the degree +1 derivation $d$ and the
degree -1 derivation $\iota_\scC(v)$.  That is,
\[
\cL_\scC(v)  := d\circ \iota_\scC(v) + \iota_\scC(v)\circ d.
\]
\end{definition}

\begin{remark} 
Note that since $d$ is $\R$-linear, so is the Lie derivative
$\cL_\scC$.   The Lie derivative is not $\scC$-linear in general.
\end{remark}

The following lemma will be used in the next section.
\begin{lemma} \label{lem:Lie_der-for-rings}
The Lie derivative map $\cL_\scC: \cDer(\scC) \to \Der^{0}
(\Lambda^\bullet\Omega^1_\scC)$ is natural in $\scC$.
\end{lemma}  

\begin{proof}
Suppose $f:\scC\to \scB$ is a map of $\cin$-rings and $v\in
\cin\Der(\scC)$, $w\in \cin\Der(\scB)$ are $f$-related.     Since
$\iota_\scC$ is natural in $\scC$ (Lemma~\ref{lem:contr-for-rings}),
$\Lambda^\bullet (f) \circ \iota _\scC(v) = \iota_\scB(w) \circ
\Lambda^\bullet (f)$.  The map $\Lambda^\bullet
(f):\Lambda^\bullet \Omega^1_\scC\to \Lambda^\bullet \Omega^1_\scB$
commutes with the differentials: $\Lambda^\bullet
(f) \circ d = d\circ \Lambda^\bullet
(f)$.  By Lemma~\ref{lem:contr-for-rings} the map $\Lambda^\bullet
(f)$ commutes with contractions as well.
Consequently
\[
\Lambda^\bullet (f) \circ \cL_\scC(v) = \Lambda^\bullet (f) \circ (
d\circ \iota_\scC(v) + \iota_\scC(v)\circ d) =  (d\circ \iota_\scB(w) +
\iota_\scB(w)\circ d) \circ \Lambda^\bullet (f)  = \cL_\scB(w) \circ \Lambda^\bullet (f).
\]
\end{proof}

We are now in position to prove Theorem~\ref{prop:cartan_calc_level0},
the existence of Cartan calculus for a $\cin$-ring $\scC$.  Again, to
reduce the clutter we drop the subscript $\, _\scC$ from the
contraction and the Lie derivative maps.

\begin{proof}[Proof of Theorem~\ref{prop:cartan_calc_level0}]

({v}) holds by definition of $\cL(v)$; this  identity is included in the
statement of the proposition to
parallel the usual   formulation of Cartan calculus.

Since $\iota(v)$, $\iota(w)$ are derivations of degree -1,  their graded
commutator $ [\iota(v), \iota(w)] $
is a derivation of degree -2.     In particular for any $\alpha
\in \Omega^1_\scC$ we have $[\iota(v), \iota(w)]\alpha \in
\Lambda^{-1} \Omega^1_\scC =0$.  Since $\Lambda^2\Omega^1_\scC$ is
generated, as an $\scC$-module,  by the set $\{\alpha \wedge \beta\mid
\alpha, \beta \in \Omega^1_\scC\}$ and since
\[
[\iota(v),\iota(w)] (\alpha \wedge \beta) =0
\]  
for all $\alpha, \beta \in \Omega^1_\scC$,
the restriction $[\iota(v),
\iota(w)]|_{\Lambda^2\Omega^1_\scC}$ is $0$.  This, in turn, implies
that $  [\iota(v), \iota(w)] |_{\Lambda^k\Omega^1_\scC} =0$
for all $k$ and thereby  proves  ({iv}).

Since $d\circ d =0$,
\[  
\cL(v) \circ d =( d\circ \iota(v) + \iota(v) \circ d)  \circ d= d\circ
\iota(v) \circ d = d\circ \cL(v),
\]
which proves  ({i}).

To prove  ({ii}) note that both sides are degree 0 derivations
that commute with $d$.   Therefore it is enough to check that both
sides agree on $\Lambda^0\Omega^1_\scC  = \scC$.  For any $a\in \scC$
and any $u\in \Der(\scC)$, $\iota(u)da = u(a)$ by the universal
property of $d:\scC\to \Omega^1_\scC$, or, if you prefer, by the
definition of $\iota$.   Consequently
\[
  \cL(u)a =  d \iota(u) a + \iota (u) da  =d( 0) + u(a) = u(a). 
\]
Therefore
\[
\cL({[v,w]})a= [v,w]a = v(w(a)) - w(v(a)) = \cL(v)(\cL(w)a) - \cL(w
)(\cL(v) a) =[\cL(v), \cL(w)] a.
\]

It remains to prove (iii).
Since $\cL(v)$ is a derivation  of degree 0 and $\iota(w)$ is a
derivation of degree $-1$ their commutator $[\cL(v) , \iota(w)]$
is a derivation of degree $-1$ hence is $\scC$-linear by
Lemma~\ref{lem:1level0}. 
The module  $\Omega^1_\scC$ of K\"ahler differentials is generated, as an $\scC$-module, by the set
$\{db\}_{b\in \scC}$. A computation shows that for any $b\in\scC$
\[
[\cL(v) , \iota(w)] (db) = v(w(b))- w(v(b)) = \iota([v,w]) db.
\]  
Hence $ [\cL(v) , \iota(w)]
|_{\Omega^1_\scC} = \iota([v,w]) |_{\Omega^1_\scC}$.  An easy 
argument shows that if any two degree -1 derivations agree on
$\Omega^1_\scC$ then they agree on all of $\Lambda^\bullet \Omega^1_\scC$,
and {(iii)} follows.
\end{proof}

\section{Cartan calculus for local $\cin$-ringed spaces} \label{sec:sheaf}

The goal of this section is to prove Theorem~\ref{main_thm}.  We start
with some technical observations.

\begin{lemma}\label{lem:derk}
Fix an integer $k\in \Z$. Suppose that for any $\cin$-ring $\scC$ we have an $\R$-linear map
 \[ 
h_\scC: \cin\Der(\scC) \to \Der^k (\Lambda^\bullet \Omega_\scC^1)
\]
which is natural in $\scC$ (Definition~\ref{def:nat}).   Then for any
local $\cin$-ringed space $(M, \scA)$ there is a unique map of
real vector spaces
\[
H_M:  \cin\Der(\scA) \to \Der^k (\Lambda^\bullet \Omega_\scA^1)
\]
with
\[
H_M (v)  = \{ h_{\scA(U)}(v_U)\}_{U\in \Open(M)}.
\]
Moreover if each map $h_\scC $ is a map of $\scC$-modules, then $H_M$
is $\scA(M)$-linear.  
\end{lemma}

\begin{proof}
For any $v\in \cin\Der(\scA)$ the set    $\{ h_{\scA(U)}(v_U)\}_{U\in
  \Open(M)}$ is an element of $\Der^k (\Lambda^\bullet \Omega_\scA^1)$
if and only if the diagram
\[
\xy
(-20,10)*+{\Lambda^\bullet\Omega^1_{\scA(U))}}="1";
(20,10)*+{\Lambda^{\bullet+1}\Omega^1_{\scA(U))}}="2";
(-20,-5)*+{\Lambda^\bullet\Omega^1_{\scA(V))}}="3";
(20,-5)*+{\Lambda^{\bullet+1}\Omega^1_{\scA(V))}}="4";
{\ar@{->}^{h_{\scA(U)}(v_U)} "1";"2"};
{\ar@{->}_{\Lambda^\bullet(^\scA\!\rho_V^U)} "1";"3"};
{\ar@{->}^{\Lambda^\bullet(^\scA\!\rho_V^U)} "2";"4"};
{\ar@{->}_{h_{\scA(V)}(v_V)} "3";"4"};
\endxy
\]
commutes   for all $V, U\in \Open(M)$ with $V\subset U$.  But $v=
\{v_U\}_{U\in \Open(M)}\in \cin\Der(\scA)$ means that
\[
\xy
(-10,10)*+{\scA(U)}="1";
(10,10)*+{\scA(U)}="2";
(-10,-5)*+{\scA(V)}="3";
(10,-5)*+{\scA(V)}="4";
{\ar@{->}^{v_U} "1";"2"};
{\ar@{->}_{^\scA\!\rho_V^U} "1";"3"};
{\ar@{->}^{^\scA \!\rho_V^U} "2";"4"};
{\ar@{->}_{v_V} "3";"4"};
\endxy
\]
commutes for all $V\subset U\subset M$.  By assumption on the 
collection of maps $\{ h_{\scA(U)}\} _{U\in \Open(M)}$
\[
h_{\scA(V)}(v_V)\circ \Lambda^\bullet(^\scA\rho_V^U) =
\Lambda^\bullet(^\scA\rho_V^U) \circ h_{\scA(U)} (v_U)
\]
for all $V\subset U$.   Therefore $H_M(v) =\{ h_{\scA(U)}(v_U)\}_{U\in
  \Open(M)}$ does lie in $\Der^k (\Lambda^\bullet \Omega_\scA^1)$.
Since each map $h_{\scA(U)}$ is $\R$-linear, $H_M$ is $\R$-linear.  If
each map $h_{\scA(U)}$ is $\scA(U)$ linear then $H_M$ is $\scA(M)$-linear.
\end{proof}  

\begin{corollary} \label{cor:4.3june} We use the notation of Lemma~\ref{lem:derk}.\\[2pt]
(a)\quad For any open set $U\subset M$ we have an $\R$-linear map
\[
H_U :  \cin\Der(\scA|_U) \to \Der^k (\Lambda^\bullet \Omega_\scA^1|_U)
\]
If each $h_\scC$ is $\scC$-linear then $H_U$ is $\scA(U)$-linear.\\[2pt]
\noindent (b)\quad For any pair $V, U\in \Open(M)$ with $V\subset U$
the diagram
\begin{equation} \label{eq:4.3.1june}
\xy
(-20,10)*+{\cin\Der(\scA|_U)}="1";
(20,10)*+{\Der^{k}(\Lambda^\bullet\Omega^1_{\scA}|_U)}="2";
(-20,-5)*+{\cin\Der(\scA|_U)}="3";
(20,-5)*+{\Der^{k}(\Lambda^{\bullet}\Omega^1_{\scA}|_V)}="4";
{\ar@{->}^{H_U} "1";"2"};
{\ar@{->}_{\rho_V^U} "1";"3"};
{\ar@{->}^{\rho_V^U} "2";"4"};
{\ar@{->}_{H_V} "3";"4"};
\endxy
\end{equation}
commutes.
\end{corollary}  
\begin{proof}
For any $U\in \Open(M)$, $\Lambda^\bullet \Omega^1_\scA|_U=
\Lambda^\bullet \Omega^1_{\scA|_U}$.  Consequently (a) follows by
applying Lemma~\ref{lem:derk} to the $\LCRS$ $(U,\scA|_U)$.

To prove (b) observe that
\[
\rho^U_V (H_U(v)) = \rho^U_V (\{ h_{\scA(W)}(v_W)\}_{W\in
  \Open(U)}) = \{ h_{\scA(W)}(v_W)\}_{W\in
  \Open(V)} = H_V (\rho^U_V (v)).
\]  
\end{proof}  

\begin{corollary} \label{cor:4.3}
Suppose that for any $\cin$-ring $\scC$ we have an $\R$-linear map
 \[ 
h_\scC: \cin\Der(\scC) \to \Der^k (\Lambda^\bullet \Omega_\scC^1)
\]
which is natural in $\scC$ (Definition~\ref{def:nat}).   Then for any
local $\cin$-ringed space $(M, \scA)$ there is a unique map of
presheaves of vector spaces
\[
\scH: \cin\Dert(\scA) \to \Dert^k (\Lambda^\bullet \Omega_\scA^1)
\]
with $\scH_U = H_U$ for each $U\in \Open(M)$. Here
\[
  H_U: \cin\Dert(\scA) (U) \equiv \cin\Der(\scA|_U) \to \Der^k
  (\Lambda^\bullet \Omega_\scA^1|_U) \equiv \Dert^k (\Lambda^\bullet
  \Omega_\scA^1)(U)
\]
are the maps constructed in Corollary~\ref{cor:4.3june}.

Moreover if each map $h_\scC$ is $\scC$-linear, then $\scH$ is a map
of presheaves of $\scA$-modules.
\end{corollary}  

\begin{proof}
We need to check that $\scH =\{H_U\}_{U\in \Open(M)}$ is a map of
presheaves.  Since the diagram \eqref{eq:4.3.1june} commutes, the
collection $\{H_U\}_{U\in \Open(M)}$ of maps of vector spaces is a map
of presheaves of vector spaces. And if each $H_U$ is $\scA(U)$-linear,
then $\scH$ is a map of $\scA$-modules. 
\end{proof}  

\begin{construction} \label{constr:1}
For a $\cin$-ring $\scC$ the maps $\iota_\scC$ and $\cL_\scC$ are
natural in $\scC$ (see Lemma~\ref{lem:contr-for-rings} and
Lemma~\ref{lem:Lie_der-for-rings}).  By Corollary~\ref{cor:4.3} for
any local $\cin$-ringed space $(M,\scA)$ we get
(unique) maps of presheaves
\[
\bi: \cin\Dert(\scA)\to \Dert^{-1} (\Lambda^\bullet
\Omega^1_\scA )  \qquad (\textrm{the {\sf contraction map}})
\]
and
\[
\scL: \cin\Dert(\scA)\to \Dert^{-1} (\Lambda^\bullet
\Omega^1_\scA ) \qquad (\textrm{the {\sf Lie derivative map}}).
\]  
\end{construction}
\begin{remark} \label{rmrk:4.5june}
Let $(M, \scA)$ be a local $\cin$-ringed space and $\bi$, $\scL$ the
contraction and Lie derivative maps, respectively, of Construction~\ref{constr:1}.
For any open set $U\in \Open (M)$, for any derivation $v
=\{v_W\}_{W\in \Open(U)}\in \cin\Dert(\scA) (U) \equiv \cin\Der(\scA|_U)$
\begin{align*}
  (\bi)_U(v) &= \{ \iota_{\scA(W)} (v_W)\}_{W\in \Open(U)}& \\
  \textrm{
                                                           and }&&\\
 (\scL)_U(v) &= \{ \cL_{\scA(W)} (v_W)\}_{W\in \Open(U)},& 
\end{align*}
where $\iota_{\scA(W)} : \cDer(\scA(W)) \to \Der^{-1} (\Lambda^\bullet \Omega^1_{\scA(W)})$
is the contraction map constructed in Lemma~\ref{lem:contr-for-rings} and 
$\cL_{\scA(W)} : \cDer(\scA(W)) \to \Der^{0} (\Lambda^\bullet \Omega^1_{\scA(W)})$ is the Lie derivative map (Definition~\ref{def:3.4june}).
These identities hold by construction of $\bi$ and $\scL$.
\end{remark}

\begin{lemma} \label{lem:4.6june}
 Let $(M, \scA)$ be a local $\cin$-ringed space.
For any $r>0$, any $k\in \Z$ two maps of presheaves $ \scG,
\scH: (\cin\Dert (\scA))^r\to \Dert ^k (\Lambda ^\bullet \Omega^1_\scA)$
are equal if and only if
\[
\scG_U = \scH_U:  (\cin\Der (\scA|_U))^r\to \Der^k (\Lambda ^\bullet \Omega^1_\scA|_U)
\]  
for all $U\in \Open (M)$ if and only if for all $v_1, \ldots, v_r \in
\cin\Dert (\scA|_U)$ and any $W\in \Open(U)$
\[
  (\scG_U(v_1,\ldots, v_r))_W = (\scH_U(v_1,\ldots, v_r))_W \in \Der^k
 \left(\left. \Lambda ^\bullet \Omega^1_\scA\right|_U (W)\right) \equiv \Der^k(\Lambda
  ^\bullet \Omega^1_{\scA(W)}).
\]  
\end{lemma}
\begin{proof}
The proof amounts to unpacking the definitions and is omitted.
\end{proof}

\begin{theorem} \label{thm:cartan-for-presheaves} Let $(M,\scA)$ be a
  local $\cin$-ringed space, $\cin\Dert(\scA)$ the sheaf of
  $\cin$-derivations of the sheaf $\scA$ and
  $\Dert^\bullet (\Lambda^\bullet \Omega^1_\scA) $ the presheaf of graded Lie
  algebras of graded derivations of the presheaf 
  $(\Lambda^\bullet \Omega^1_\scA , \wedge, d)$  of $\cin$-algebraic differential forms,
  ${ad}(d): \Dert^\bullet (\Lambda^\bullet \Omega^1_\scA) \to
  \Dert^{\bullet +1} (\Lambda^\bullet \Omega^1_\scA)$ the corresponding map between
  graded derivations,
  $\bi :\cin\Dert(\scA)\to \Dert^{-1}
  (\Lambda^\bullet \Omega^1_\scA)$ the contraction map and
  $ {\scL} :\cin\Dert(\scA)\to \Dert^{0} (\Lambda^\bullet \Omega^1_\scA)$
  the Lie derivative map.              
  Then 
  \begin{eqnarray}
    {ad}(d)\circ {\bi} &= &  {\scL} \label{eq:4.7.1}\\
  {ad}(d)\circ  {\scL} &=&0 \label{eq:4.7.2}\\
   {\scL} \circ [\cdot, \cdot] &= &[\cdot, \cdot] \circ ( {\scL}\times
  { \scL})
                                    \label{eq:4.7.3}\\
 \, [\cdot, \cdot] \circ ( {\scL}\times {\bi}) &= & {\bi} \circ   [\cdot,
  \cdot] 
                                                       \label{eq:4.7.4}\\
\, [\cdot, \cdot] \circ ( {\bi} \times{\bi})
                                          &=& 0.
                                              \label{eq:4.7.5}
\end{eqnarray}    

\end{theorem}
\begin{proof}
The theorem is an easy consequence of Remark~\ref{rmrk:4.5june},
Lemma~\ref{lem:4.6june} and Cartan calculus for $\cin$-rings,
Theorem~\ref{prop:cartan_calc_level0}.  We prove that
\eqref{eq:4.7.1} holds.  The proofs of the rest of the equalities
are similar.

Both sides of the equation are maps of presheaves, which are equal if
and only if their $U$-components $(ad(d)\circ \bi)_U= (ad(d))_U\circ
(\bi)_U$ and $(\scL)_U$ are equal for all $U\in \Open(M))$.  These
components are maps of presheaves from $\cin\Der(\scA|_U)$ to $\Der^0
(\Lambda^\bullet \Omega^1_\scA|_U)$.  By Remark~\ref{rmrk:2.90},
$(ad(d))_U = \{ad(d_{\scA(W)}\}_{W\in \Open(U)}$.  By definitions of
$\bi $ and $\scL$ (Construction~\ref{constr:1} and Remark~\ref{rmrk:4.5june})
\[
  (\bi)_U(v) = \{ \iota_{\scA(W)} (v_W)\}_{W\in \Open(U)} ,  \qquad
  (\scL)_U(v) = \{ \cL_{\scA(W)} (v_W)\}_{W\in \Open(U)}
\]
for all $v =\{v_W\}_{W\in \Open(U)}\in \cin\Dert(\scA) (U) \equiv
\cin\Der(\scA|_U)$.
Hence
\begin{align*}
 \left( (ad(d))_U\circ (\bi)_U \right)(v)&= \{ad(d_{\scA(W)} ( \iota_{\scA(W)}
  (v_W))\}_{W\in \Open(U)} = \{[d_{\scA(W)} ,\iota_{\scA(W)}
  (v_W)]\}_{W\in \Open(U)} \\ &= \{ \cL_{\scA(W)} (v_W)\}_{W\in \Open(U)}
  =  (\scL)_U(v),
\end{align*}
and \eqref{eq:4.7.1} holds.
\end{proof}
\mbox{}\\
We will use  Theorem~\ref{thm:cartan-for-presheaves}  to prove
Theorem~\ref{main_thm}.   Our proof requires a few facts about
sheafification which we now recall.

We continue to fix a local $\cin$-ringed space $(M, \scA)$.  For each
open subset $U$ of $M$ we have have the category $\Psh_U$ of
presheaves of modules over $(U,\scA|_U)$ and the full subcategory
$\Sh_U$ of $\Psh_U$ of sheaves.  The inclusion functor
$\Sh_U\hookrightarrow \Psh_U$ has a left adjoint, the sheafification
functor, which we denote by $\sh_U$ (or just by $\sh$ when the space
$U$ is clear).  The sheafification functor is unique up to a unique
isomorphism.   We also have restriction functors
\[
\rho^U_V:\Psh_U\to \Psh_V
\]  
for all open sets $V,U\in \Open(M)$ with $V\subset U$.  To reduce
notational clutter we  set
\[
\scN_1|_V\xrightarrow{X|_V} \scN_2|_V := \rho^U_V\, (\scN_1\xrightarrow{X}\scN_2)
\]  
for all morphisms $\scN_1\xrightarrow{X}\scN_2$ in $\Psh_U$.
Note that for any three open sets $W\subset V \subset U\subset M$
\begin{equation} \label{eq:4.7.6june} 
\rho^V_W\circ \rho^U_V = \rho ^U_W,
\end{equation}  
i.e., $\left(\scN|_V\right)|_W = \scN|_W$ and so on.
\begin{lemma} \label{lem:4.8june}
  Let  $(M, \scA)$ be a local $\cin$-ringed space.  We can
  define a family $\{\sh_U:\Psh_U\to \Sh_U\}_{U\in\Open(M)}$ of
  sheafification functors so that for two modules $\scN_1, \scN_2\in
  \Psh_M$ any pair $V\subset U\subset M$ of open sets  and any
  morphism $X\in \Hom(\scN_1|_U, \scN_2|_U)$
\[  
 \sh_U(X) |_V = \sh_V (X|_V).
\] 
\end{lemma}  

\begin{proof}
Recall that the value of a sheafification functor $\sh_U:\Psh_U\to \Sh_U$  on
morphisms is uniquely determined by its value on objects and choices
of universal arrows $\{\eta_U^\scN :\scN\to \sh(\scN)\}_{\scN\in
  \Psh_U}$.   Recall also that given any morphism of presheaves $f$
from a presheaf $\scN\in \Psh_U$ to a sheaf $\scS$ so that $f$ induces an
isomorphism on stalks, we may consider $\scS$ as a sheafification of
$\scN$ and $f$  as the universal arrow $\eta_U^\scN$.

Given $U\in \Open(M)$ we now {\em choose} a functor $\sh_U$ in such a way
that for any $\scN\in \Psh_M$,
\[
\sh_U(\scN):= \sh(\scN)|_U\quad \textrm{ and } \quad
\eta_U^{(\scN|_U)} = \eta^\scN_M|_U.
\]  
Then for any $\scN_1, \scN_2\in \Psh_M$ and any
$X\in \Hom(\scN_1|_U, \scN_2|_U)$ the diagram
\[
\xy
(-25,10)*+{\scN_1|_U}="1";
(15,10)*+{\sh(\scN_1)|_U}="2";
(-25,-5)*+{\scN_2|_U}="3";
(15,-5)*+{\sh(\scN_2)|_U}="4";
{\ar@{->}^{(\eta^{\scN_1}_M)|_U} "1";"2"};
{\ar@{->}_{X} "1";"3"};
{\ar@{->}^{\sh_U(X)} "2";"4"};
{\ar@{->}_{(\eta^{\scN_2}_M)|_U} "3";"4"};
\endxy
\]
commutes.  Applying the restriction functor $\rho^U_V$ to the diagram
above (and taking \eqref{eq:4.7.6june} into account) gives  the commutative diagram
\[
\xy
(-25,10)*+{\scN_1|_V}="1";
(15,10)*+{\sh(\scN_1)|_V}="2";
(-25,-5)*+{\scN_2|_V}="3";
(15,-5)*+{\sh(\scN_2)|_V}="4";
{\ar@{->}^{(\eta^{\scN_1}_M)|_V} "1";"2"};
{\ar@{->}_{X|_V} "1";"3"};
{\ar@{->}^{\sh_U(X)|_V} "2";"4"};
{\ar@{->}_{(\eta^{\scN_2}_M)|_V} "3";"4"};
\endxy.
\]
But 
\[
\xy
(-25,10)*+{\scN_1|_V}="1";
(15,10)*+{\sh(\scN_1)|_V}="2";
(-25,-5)*+{\scN_2|_V}="3";
(15,-5)*+{\sh(\scN_2)|_V}="4";
{\ar@{->}^{(\eta^{\scN_1}_M)|_V} "1";"2"};
{\ar@{->}_{X|_V} "1";"3"};
{\ar@{->}^{\sh_V(X|_V)} "2";"4"};
{\ar@{->}_{(\eta^{\scN_2}_M)|_V} "3";"4"};
\endxy
\]
commutes as well.  Therefore
\[  
 \sh_U(X) |_V = \sh_V (X|_V).
\] 
\end{proof}

\begin{lemma} \label{lem:4.9june}
Let $(\scM^\bullet, \wedge)$ be a presheaf of commutative graded
algebras over a local $\cin$-ringed space $(M,\scA)$ and
$(\sh_M(\scM^\bullet), \sh_M(\wedge))$ its sheafification.   The family of
of sheafification functors $\{\sh_U\}_{U\in \Open(M)}$ constructed in
Lemma~\ref{lem:4.8june} give rise to a degree-preserving map
\[
G: \Dert^\bullet (\scM^\bullet) \to \Dert^\bullet (\sh_M(\scM^\bullet)) 
\]
of presheaves of graded Lie algebras.
\end{lemma}

\begin{proof}
Fix a degree $k\in \Z$.  For any $U\in \Open(M)$ for any integer
$\ell\in \Z$ we have a map
\begin{align*}
\sh_U: \Hom (\scM^\ell|_U, \scM^{\ell +k}|_U) &\to \Hom
  (\sh_U(\scM^\ell|_U), \sh_U(\scM^{\ell +k}|_U) )\\
 & \left( \equiv \Hom (\sh_M(\scM^\ell)|_U, \sh_M(\scM^{\ell +k})|_U) \right)
\end{align*}
of $\scA(U)$-modules.  Since $\sh_U$ is a functor, for any derivation
$X= \{X^\ell\}_{\ell\in \Z}\in \Der^k (\scM^\bullet |_U)$
\[
G_U (X) := \{\sh_U(X^\ell)\}_{\ell\in \Z}
\]
is a derivation of degree $k$ of the CGA $(\sh_M(\scM^\bullet)|_U, \sh(\wedge))$.
We thus get a map
\[
G_U :\Der^k (\scM^\bullet|_U) \to \Der^k (\sh_M(\scM^\bullet)|_U)
\]  
of $\scA(U)$-modules.

By Lemma~\ref{lem:4.8june} for any pair $V,U\subset M$ of open sets
with $V\subset U$ and any $X\in \Der^k (\scM^\bullet|_U)$
\[
G_V (X|_V) = \{\sh _V (X^\ell|_V)\}_{\ell \in \Z} =  \{\sh _U (X^\ell)
|_V)\}_{\ell \in \Z}
= G_U(X)|_V.
\]
Hence $G= \{G_U\}_{U\in \Open(M)}$ is a well-defined map of presheaves
of $\scA$-modules from $\Dert^k (\scM^\bullet)$ to $\Dert^k
(\sh_M(\scM^\bullet))$.  Since $k$ is arbitrary we get a degree
preserving map from the graded presheaf $\Dert^\bullet (\scM^\bullet)
$ to the graded presheaf $\Dert^\bullet (\sh_M(\scM^\bullet) )$ of
$\scA$-modules.

It remains to check that $G$ preserves (graded) commutators.  Since
$G$ preserves degrees, 
it is enough to check that for each open set $U$, the map $G_U:
\Der^\bullet(\scM^\bullet|_U)\to \Der^\bullet(\sh_M(\scM^\bullet)|_U)$
of graded modules preserves composition.
This is not hard and follows from the fact that each $\sh_U$ is a
functor.  Namely,  given $X=\{X^\ell\}_{\ell\in \Z} \in
\Der^{|X|}(\scM^\bullet|_U)$, $Y=\{Y^\ell\}_{\ell\in \Z} \in
\Der^{|Y|}(\scM^\bullet|_U)$,
\[
  G_U (X\circ Y) = \{\sh_U (X^{\ell+|Y|} \circ Y^\ell)\}_{\ell\in \Z} =
  \{\sh_U (X^{\ell+|Y|} )\circ \sh_U( Y^\ell)\}_{\ell\in \Z}  = G_U(X)
  \circ G_U(Y).
\]  
\end{proof}

\begin{lemma} \label{lem:4.10june}
  
Let $(\scM^\bullet, \wedge, d)$ be a presheaf of commutative
differential graded algebras over a local $\cin$-ringed space
$(M,\scA)$ and $G:\Dert^\bullet (\scM^\bullet) \to \Dert^\bullet
(\sh_M(\scM^\bullet)) $ the map of graded presheaves of
Lemma~\ref{lem:4.9june} and $ad(d): \Dert^\bullet (\scM^\bullet) \to
\Dert^{\bullet +1} (\scM^\bullet)$, $ad(\sh(d)): \Dert^\bullet (\sh(\scM^\bullet)) \to
\Dert^{\bullet +1}(\sh(\scM^\bullet ))$ the maps of graded presheaves
of real vector spaces induced by the derivation $d:\scM^\bullet \to
\scM^{\bullet +1}$ and its sheafification $\sh(d) \equiv G_M(d)$ (cf.\
Proposition~\ref{prop:ad(x)}). 
The diagram
\begin{equation} \label{eq:5.9june}
\xy
 (-25,10)*+{\Dert^\bullet (\scM^\bullet) }="1";
 (25,10)*+{\Dert^{\bullet +1}(\scM^\bullet )}="2";
 (-25,-5)*+{\Dert^\bullet (\sh (\scM^\bullet))}="3";
(25,-5)*+{\Dert^{\bullet +1}(\sh((\scM^\bullet))}="4";
 {\ar@{->}^{{ad}(d)} "1";"2"};
 {\ar@{->}_{G} "1";"3"};
 {\ar@{->}^{G} "2";"4"};
{\ar@{->}_{{ad}(\sh(d))} "3";"4"};
\endxy
\end{equation}
commutes.
\end{lemma}  

\begin{proof}
By Proposition~\ref{prop:ad(x)}, for any degree $k$ derivation $X$ of
a presheaf $(\scN^\bullet, \wedge)$ of commutative graded algebras
there is a map $ad(X):\Dert^\bullet (\scN^\bullet) \to
\Dert^{\bullet+k} (\scN^\bullet)$ of graded presheaves of real vector
spaces increasing degrees by $k$.  In particular, for any open set
$U\in \Open(M)$ the $U$-component $ad(X)_U$ of $ad(X)$ is given by
\[
ad(X)_U Y = [X|_U, Y] \equiv X|_U\circ Y - (-1)^{k|Y|} Y\circ X|_U.
\]  
By construction of the map $G$, the value of $\sh \equiv
\sh_M:\Psh_M\to \Sh_M$ on the differential $d:\scM^\bullet \to
\scM^{\bullet +1}$ is $G_M (d)$.  Since for any $U\in \Open(M)$, $G_U: \Der^\bullet (\scM^\bullet|_U) \to \Der^\bullet
(\sh_M(\scM^\bullet)|_U)$ is a map of graded Lie algebras,
\[
G_U ([d|_U, Y]) = [G_U(d|_U), G_U(Y)]
\]
for all $Y \in \Der^\bullet (\scM^\bullet|_U) $.
Hence
\begin{align*}
(ad (G_M(d))_U \circ G_U)\, Y &= [G_M(d)|_U, G_U(Y)] =[G_U (d|_U),
                              G_U(Y)] \\&= G_U ([d|_U, Y]) = (G_U \circ ad(d)_U))\, Y.
\end{align*}  
Therefore 
\[
ad (\sh(d)) _U \circ G_U  = G_U \circ ad (d)_U
\]  
for all $U$ and the diagram \eqref{eq:5.9june} commutes.
\end{proof}

\begin{notation}
We now simplify (and abuse) our  notation.  We  write $\sh$ for the
map $G$ constructed in Lemma~\ref{lem:4.9june}.  In particular since
the sheaf $\bOmega^\bullet
_\scA$ of $\cin$-algebraic de Rham forms is the
sheafification of the presheaf $(\Lambda^\bullet
\Omega^1_\scA, \wedge, d)$ of CDGAs we write
\[
  \sh: \Dert^\bullet  (\Lambda^\bullet \Omega^1_\scA) \to
  \Dert^\bullet (\bOmega^\bullet _\scA)
\]
for the map of presheaves of graded Lie algebras induced by the
sheafification functors.
\end{notation}  

\begin{corollary}
For any local $\cin$-ringed space $(M, \scA)$ the diagram
\begin{equation} \label{eq:4.12.1}
\xy
(-25,10)*+{\Dert^k(\Lambda^\bullet \Omega^1_\scA)}="1";
(15,10)*+{\Dert^{k+1}(\Lambda^\bullet \Omega^1_\scA)}="2";
(-25,-5)*+{\Dert^k(\bOmega^\bullet _\scA)}="3";
(15,-5)*+{\Dert^{k+1}(\bOmega^\bullet _\scA)}="4";
{\ar@{->}^{{ad}(d)} "1";"2"};
{\ar@{->}_{\sh} "1";"3"};
{\ar@{->}^{\sh} "2";"4"};
{\ar@{->}_{{ad}(\sh(d))} "3";"4"};
\endxy
\end{equation}
commutes.
\end{corollary}  
\begin{proof}
Apply Lemma~\ref{lem:4.10june} to the presheaf $(\Lambda^\bullet
\Omega^1_\scA, \wedge, d)$ of CDGAs  and keep in mind that the sheaf $\bOmega^\bullet
_\scA$ is the sheafification of $\Lambda^\bullet \Omega^1_\scA$.
\end{proof}

\begin{definition} Let $(M, \scA)$ be a local $\cin$-ringed space. We
  define the {\sf Lie derivative map}
  \[
  \frL: \cin\Der(\scA) \to \Dert^0(\bOmega^\bullet _\scA)
  \]
   on the
  level of sheaves to be the composite of $\sh$ and $\scL$:
\[\frL:= \left(\Dert^0(\bOmega^\bullet _\scA) \xleftarrow{\sh}
\Dert^0(\Lambda^\bullet \Omega^1_\scA)\right)\circ \left( \Dert^0(\Lambda^\bullet
\Omega^1_\scA) \xleftarrow{\scL} \cin\Der(\scA)\right).
\]
We define the {\sf contraction map } 
\[\fri:
  \cin\Der(\scA) \to \Dert^{-1} (\bOmega^\bullet _\scA)
 \]  to be the
  composite of $\sh$ and $\bi$:
\[  
\bfi := \left(\Dert^{-1}(\bOmega^\bullet _\scA) \xleftarrow{\sh}
\Dert^{-1}(\Lambda^\bullet \Omega^1_\scA)\right)\circ \left( \Dert^{-1}(\Lambda^\bullet
\Omega^{-1}_\scA) \xleftarrow{\bi} \cin\Der(\scA) \right).
\] 
\end{definition}

\begin{proof} [Proof of Theorem~\ref{main_thm}]
To prove \eqref{eq:5.9j} we observe that
\[
\frL =\sh \circ \scL \stackrel{\eqref{eq:4.7.1} }{=} \sh \circ ad(d)
\circ \bi \stackrel{ \eqref{eq:4.12.1} }{=} ad(\sh(d)) \circ \sh \circ
\bi = ad (\sh(d)) \circ \bfi.
\]
By Lemma~\ref{lem:2.33june}, $ad(\sh (d))\circ ad(\sh(d)) =0$.
Therefore
\[
ad(\sh(d))\circ \frL \stackrel{\eqref{eq:5.9j}}{=} ad(\sh (d))\circ
ad(\sh(d)) \circ \bfi = 0\circ \bfi =0,
\]  
which proves \eqref{eq:5.10j}.

To prove \eqref{eq:5.11j} we recall that by Lemma~\ref{lem:4.9june}
the map $\sh \, (=G): \Dert^\bullet (\Lambda^\bullet \Omega^1_\scA)
\to \Dert^\bullet (\bOmega ^\bullet_\scA)$ is a map of presheaves of
graded Lie algebras.  If we view the graded commutators $[\cdot,
\cdot]_{\Dert^\bullet (\Lambda^\bullet \Omega^1_\scA)}$ and $[\cdot,
\cdot]_{\Dert^\bullet (\bOmega ^\bullet_\scA)}$ as maps of graded presheaves
then Lemma~\ref{lem:4.9june} says that
\begin{equation} \label{eq:4.13.1}
\sh \circ [\cdot, \cdot] = [\cdot, \cdot] \circ (\sh \times \sh),
\end{equation}
where we dropped the subscripts on the commutators to reduce the
clutter.  Consequently
\[
\frL \circ [\cdot, \cdot] = \sh \circ \cL \circ [\cdot,
\cdot]\stackrel{\eqref{eq:4.7.3}}{=} \sh \circ [\cdot, \cdot ]\circ
(\cL\times \cL)
\stackrel{\eqref{eq:4.13.1}}{=}  [\cdot, \cdot ] \circ (\sh\times \sh)
\circ (\cL \times \cL) = 
[\cdot, \cdot ] \circ (\frL\times \frL).
\]
This proves \eqref{eq:5.11j}.
Similarly,
\[
[\cdot, \cdot] \circ ( \frL\times \bfi) = [\cdot, \cdot] \circ (
\sh\circ \scL\times \sh\circ \bi) = \sh\circ  [\cdot, \cdot] \circ
(\scL\times \bi) \stackrel{\eqref{eq:4.7.4}}{=} \sh\circ \bi = \bfi,
\]  
which proves \eqref{eq:5.12j}.
Finally,
\[
 [\cdot, \cdot] \circ ( \bfi \times\bfi) =  [\cdot, \cdot] \circ ( \sh
 \circ \bi \times\sh\circ \bi) = \sh \circ  [\cdot, \cdot] \circ (\bi
 \times \bi) \stackrel{\eqref{eq:4.7.5}}{=} \sh \circ 0 = 0.
\]  
\end{proof}

\section{Some examples} \label{sec:examples}

There are (non-empty) local $\cin$-ringed spaces $(M,\scA)$ with the
property that the sheaf $\cin\Dert(\scA)$ of $\cin$-derivations is
zero.  In this case the Theorem~\ref{main_thm} is a statement about
zero maps, which is not terribly interesting.  The simplest example,
in effect, has been known since 1950's \cite{NW}: given a topological
space $M$, the space of algebraic derivations of the algebra $C^0(M)$
of continuous functions is zero.  We recall the detail of the argument
below in Lemma~\ref{lem:NW}.  Note that $(M, C^0(M))$ is a
differential space (provided $M$ is completely regular).  Given this
fact the reader may wonder if the only nontrivial examples are smooth
manifolds with the standard Cartan calculus.  Fortunately this is not
the case.  One class of examples come from closed subsets of Euclidean
spaces. In particular this class of examples includes compact
regions in Euclidean spaces with nonsmooth boundary, such as the ones
considered in \cite{Simsek}. %

We start by proving a lemma that is essentially due to Newns and
Walker \cite{NW} (it is not stated this way in their paper but what
they prove amounts to the same thing).
\begin{lemma} \label{lem:NW}
Let $M$ be a topological space and $\scC= C^0(M)$ the $\cin$-ring of
continuous functions. Then $\cDer (\scC) = 0$.
\end{lemma}

\begin{proof}
We argue that any $\R$-algebra derivation $v:\scC \to \scC$ is zero:
for any $f\in \scC$ and for any $x\in M$, $v(f)\,(x) = 0$.  Fix $x\in
M$.  Since 
\[
f = f_+ - f_-
\]
where $f_+ = \max(f, 0)$ and $f_- = \max(-f, 0)$, we may assume that
$f$ is nonnegative.  Since $v$ is a derivation $v(c)=0$ for any
constant function $c$.  Therefore $v(f) = v(f - f(x))$, and it is no
loss of generality to assume that $f(x) = 0$.  Since $f\geq 0$, it has
a nonnegative square root $\sqrt{f}$.  Since $f(x)=0$,  $\sqrt{f} (x)
= 0$ as well.    We now compute:
\[
v(f)\, (x) = v(\sqrt{f}\sqrt{f})\, (x) = \left( 2 \sqrt{f} \,v(\sqrt{f})
\right) \, (x) = 2\cdot 0 \cdot  v(\sqrt{f}) (x) = 0.
\]  
\end{proof}
\noindent
The following lemma will be useful in computing the promised nontrivial examples.
\begin{lemma} \label{lem:5.1}
Let $M\subset \R^m$ be a closed subset and $(M, \scF:= \cin(\R^m)|_M)$
the corresponding differential space (cf.\ Remark~\ref{rmrk:2.61}).
Let
\[
  I:= \{f\in \cin(\R^m)\mid f|_M = 0\},
\]
\[
  \scD:= \{V\in \chi(\R^m)\mid V(I)\subseteq I\}
\]
(here as before $\chi(\R^m)$ is the module of vector
fields on $\R^m$) and
\[
\scJ:= \{V\in \scD \mid V(x_i) \in I, 1\leq
i\leq m\} = \{V\in \scD \mid V(x_i)|_M = 0 \textrm{ for all }i\},
\]
where $x_1,\ldots, x_m:\R^m\to \R$ are the coordinate functions.
Then the $\scF$-module of $\cin$-derivations $\cDer(\scF)$ is
isomorphic to $\scD/\scJ$ (as an $\scF$-module).
\end{lemma}


\begin{proof}
Note that the $\cin$-ring structure on $\scF:= \cin(\R^m)|_M$ is given
by
\[
h_\scF(f_1|_M,\ldots, f_k|_M) := h \circ (f_1|_M,\ldots, f_k|_M)
\]  
for all $k\geq 0$, all $h\in \cin(\R^k)$ and all $f_1|_M,\ldots,
f_k|_M \in \cin(\R^m)|_M$ (on the right we view $(f_1|_M,\ldots,
f_k|_M)$ as a map from $M$ to $\R^k$).

Given $V\in \scD$, $V(f) \in I$ for all $f\in I$ by definition of
$\scD$.  Consequently
\[
\pi(V): \scF\to \scF, \qquad \pi(V)\,(f|_M):= V(f)|_M
\]  
is a well-defined map.  Since $V$ is a $\cin$-derivation, so is
$\pi(V)$.   We get a map $\pi:\scD \to \cDer(\scF)$. It is easy to
check that $\pi$ is $\R$-linear.  We now argue that $\pi$ is
surjective.

Given $v\in \cDer(\scF)$ there exist $a_i \in \cin(\R^m)$, $1\leq i\leq m$,
so that $a_i|_M = v(x_i|_M)$.  Let 
\[V = \sum_i a_i
\frac{\partial}{\partial x_i}.
\]
Then for any $f\in \cin(\R^m)$
\begin{align*}
  V(f)|_M &= \left( \sum _i a_i \partial_i
    f\right)|_M = \sum_i a_i|_M \cdot \partial _i f|_M = \sum_i \partial _i
  f|_M \cdot v(x_i|_M)\\& = \sum_i (\partial_i f)_\scF(x_1|_M,\ldots,
  x_m|_M) \cdot v(x_i|_M) \\&= v(f_\scF(x_1|_M, \ldots, x_m|_M)) \qquad \textrm{(since $v$ is a $\cin$-derivation)}\\
  &= v(f|_M),
\end{align*}
and then $\pi(V) = v$, and $\pi$ is onto.

We next argue that $\ker \pi = \scJ$.  If $V\in \ker\pi$ then
\[
0 = \pi(V) (x_i|_M) = V(x_i)|_M
\]  
for all $i$.   Hence $V\in
\scJ$ by definition of $\scJ$.  Conversely suppose $V\in \scJ$.  Then
for any $f\in \cin(\R^m)$
\[
\pi(V) \, (f|_M) = V(f)|_M  = \left( \sum_i \partial _i f \cdot
  V(x_i)\right)|_M  = \sum_i \partial _i f|_M \cdot 
  V(x_i)|_M = 0.
\]
 Therefore $\pi(V) = 0$ and $V\in \ker \pi$.  We conclude that $\ker
 \pi =\scJ$ and that
 \[
\overline{\pi} :\scD/\scJ \to \cDer (\scF),\qquad \overline{\pi}
(V+\scJ) = \pi (V) + \scJ
 \]  
 is an isomorphism of real vector spaces.

 For any $b\in I$, $V\in \scD$, $bV\in \scJ$.  This is because for any $i$,
 \[
(bV)(x_i)|_M = b|_M\cdot  V(x_i)|_M = 0 \cdot V(x_i)|_M = 0.  
\]   
Consequently $\scD/\scJ$ is an $\cin(\R^m)/I \simeq \scF$-module with
\[
g|_M  \cdot (V+\scJ) := gV +\scJ
\]  
for all $g|_M\in \scF$ and all $V\in \scD$.

Finally we check that $\overline{\pi} $ is $\scF$-linear.   First of all, for any $g, f \in \cin(\R^m)$
\[
\pi (gV) (f|_M) = \left((gV)(f)\right) |_M = g|_M \cdot V(f)|_M =\left( g|_M\cdot
  \pi (V)\right) (f|_M),
\]  
Hence $\pi (gV) = g|_M\pi(V)$.  Secondly
\[
\overline{\pi}(g|_M \cdot (V+\scJ)) = \overline{\pi} (gV+\scJ) =
\pi(gV) = g|_M \cdot \pi(V) = g|_M \cdot \overline{\pi}( V+\scJ).
\]  
\end{proof}  

\begin{remark} It is plausible that the quotient $\scD/\scJ$ should,
  in general, be nonzero.  We will explicitly compute the quotient in
  two cases: $M$ is the closure of an open subset $W\subset \R^m$ and
  $M = \{(x,y)\in \R^2\mid xy =0\}$.
\end{remark}

\begin{example} \label{ex:5.3}
Let $M = \{(x,y)\in \R^2\mid xy =0\}$ and $\scF = \cin(\R^2)|_M$.   We
argue that $\cDer (\scF)$ is isomorphic, as a real vector space, to
$\cin(\R)\times \cin(\R)$.  We use the notation of Lemma~\ref{lem:5.1}.
Then in our case $I= \{f\in \cin(\R^2)\mid f|_M = 0\}$.  Applying
Hadamard's lemma twice (\!\cite[p.\ 17]{Jet}) we see that the ideal $I$
is generated by one function, namely by $xy$:
\[
I = \langle xy \rangle.
\]  
We next argue that
\[
\scD = \{xa_1 \frac{\partial}{\partial x} + y a_2
\frac{\partial}{\partial y} \mid a_1, a_2 \in \cin(\R^2)\}
\]
If $V = xa_1 \frac{\partial}{\partial x} + y a_2
\frac{\partial}{\partial y}$ for some $a_1, a_2 \in \cin(\R^2)$ then
$V(xy) = xy a_1 + xy a_2 \in \langle xy \rangle$, hence $V\in \scD$.
Conversely suppose $V = b_1 \frac{\partial}{\partial x} + b_2
\frac{\partial}{\partial y} \in \scD$.  Then there is $g\in
\cin(\R^2)$ so that
\[
xy g = V(xy) = y b_1 + x b_2.
\]  
This implies that $0 = y b_1 (0, y)$ for all $y$.  Hence $b_1(0,y) =
0$ for all $y\not = 0$.  By continuity $b_1(0,y) = 0$ for all $y$.
Hadamard's lemma now implies that $b_1(x,y) = x a_1(x,y)$ for some
$a_1 \in \cin(\R^2)$.  Similarly $b_2(x,y) = y a_2(x,y)$ for some $a_2
\in \cin(\R^2)$.  

Similar argument shows that
\[
\scJ = \{ xy b_1 \frac{\partial}{\partial x} + xy b_2
\frac{\partial}{\partial y} \mid b_1, b_2 \in \cin(\R^2)\}.
\]  
We conclude that for $M = \{xy =0\}$ and $\scF = \cin(\R^2)|_M$,
\begin{align*}
\cDer (\scF) &= \scD/\scJ\\ & =  \{xa_1 \frac{\partial}{\partial x} + y a_2
\frac{\partial}{\partial y} \mid a_1, a_2 \in \cin(\R^2)\}/ \{ xy b_1 \frac{\partial}{\partial x} + xy b_2
\frac{\partial}{\partial y} \mid b_1, b_2 \in \cin(\R^2)\} \\& \simeq
\left(\cin(\R^2)/ y \cin(\R^2) \right)\times \left( \cin(\R^2)/ x
  \cin(\R^2) \right).
\end{align*}
Finally $\cin(\R) \simeq \cin(\R^2)|_{\{y = 0\}} \simeq
\cin(\R^2)/y\cin(\R^2)$ and similarly $\cin(\R)\simeq  \cin(\R^2)/ x
\cin(\R^2)$.   We conclude that
\[
\cDer \left(\cin(\R^2)|_{\{xy =0\}} \right)\simeq \cin(\R)\times \cin(\R).
\]
This makes sense geometrically: since $\pi: \scD \to \cDer(\scF)$ is
surjective, any derivation of $\scF$ extends (non-uniquely) to a vector
field on $\R^2$ that vanishes at $(0,0)$ and is tangent to coordinate
axes.  Hence we can identify ``vector fields'' on $\{xy=0\}$ with
pairs of vector fields on the coordinate axes that vanish at the
origin. The space of all such pairs is (isomorphic to) $\cin(\R)^2$.
 \end{example}

\begin{example}
Let $W\subset \R^m$ be an open set and $M = \overline{W}$ its closure.
As before $\scF = \cin(\R^m)|_M$.  We argue that $\cDer(\scF) $ is
isomorphic to $\{V|_M \mid V\in \chi(\R^m)\}$, where as before
$\chi(\R^m)$ is the space of vector fields on $\R^m$ (thought both as
$\cDer(\cin(\R^m))$ and as $\cin(\R^m, \R^m)$).

For $f\in \cin(\R^m)$, if $f|_M = 0$ then $f|_W= 0$ as well.  Since $W$ is open, for any
vector field $V\in \chi(\R^m)$, $f|_W = 0$ implies that $V(f)|_W =0$
and, by continuity, that $V(f)|_M =0$.
Hence $\scD = \chi(\R^m)$.

On the other hand for any $V =\sum a_i \frac{\partial }{\partial
  x_i}$, $V(x_i) = a_i$.  Hence $V(x_i)|_M = 0$ for all $i$ implies
that $a_i|_M = 0$ for all $i$ and therefore $V|_M =0$.   Thus
\[
\scJ = \{V\in \chi(\R^m) \mid V|_M = 0\}.
\]  

Therefore, by Lemma~\ref{lem:5.1}
\[
\cDer(\scF) = \scD/\scJ = \chi(\R^m)/\{V\in \chi(\R^m) \mid V|_M = 0\}
\simeq \{ V|_M \mid V\in \cin(\R^m, \R^m)\}.
\]  
\end{example}

\begin{remark}  The example above generalizes from closed subsets of
  $\R^m$ with dense interior to closed subsets of arbitrary manifolds
  with dense interior. Namely, 
let $N$ be a manifold, $W\subset N$ an open subset and $M =
\overline{W}$ the closure of $W$ in $N$.  
Then $\cDer(\cin(N)|_M)$ is isomorphic to the vector space of restrictions
of sections of the tangent bundle $TN$ to $M$:
\[
\cDer(\cin(N)|_M) \simeq \{X|_M\mid X \in \Gamma (TN)\}.
\]
A proof of this remark will take us too far afield and we omit it.
\end{remark}  
\appendix

\section{Vector fields on
  differential spaces}  \label{app:A}

The goal of the appendix is to prove the following theorem.

\begin{theorem} \label{thm:a1}
Let $(M, \scF)$ be a differential space, $\scF_M$ the associated sheaf
of $\cin$-rings, $\cin\Der(\scF)$ the Lie algebra of derivations of
the $\cin$-ring $\scF$ (cf.\ Lemma~\ref{lem:cder-lie}) and $\cin\Der(\scF_M)$ the Lie algebra of
derivations of the sheaf $\scF_M$ (cf.\ Lemma~\ref{lem:level03.6} ).
The global sections functor
\[
\Gamma: \cin\Der(\scF_M) \to \cin\Der(\scF), \qquad
\Gamma(\{v_U\}_{U\in\Open(M)}) = v_M
\]
is an isomorphism of Lie algebras.

In particular for any $\cin$-derivation $w:\scF\to \scF$ there is a
unique family of derivations $\{v_U\}_{U\in \Open(M)}$ with $v_U\in
\cin\Der(\scF_M(U)$ and the property that for any pair of open sets
$V, U$ with $V\subseteq U$ and any $h\in \scF_M(U)$
\[
v_V(h|_V) = \left. v_U(h) \right|_V.
\]  
\end{theorem}

\begin{remark} For any subset $Y\subset M$ of a differential space
$(M, \scF)$ there is an induced differential structure $\scF^Y$ on $Y$
--- see \cite{Sn}.  It is the smallest differential structure on $Y$
containing the set $\scF|_Y$ of restrictions of the functions in
$\scF$.  If the set $Y$ happens to be open, then the induced
differential structure $\scF^Y$ is (isomorphic to) the value of the
sheaf $\scF_M$ on $Y$: $\scF^Y = \scF_M(Y)$.  Moreover it is not too
hard to show that if $Z\subset Y$ then $\scF^Z = (\scF^Y)^Z$; see, for
example, \cite{KL}.  Consequently Theorem~\ref{thm:a1} implies that
for any pair of open subsets $V, U\in \Open(M)$ with $V\subset U$ we
have well-defined Lie algebra maps
\[
\rho^U_V: \cDer (\scF_M(U)) \to \cDer (\scF_M(V)).
\]  
Namely, given $v\in \cDer (\scF_M(U)) = \cDer(\scF^U)$ there is a
unique family $\{v_W\}_{W\in \Open(U)} \in \cDer ((\scF^U)_U)$ (the
clumsy notation $(\scF^U)_U$ stands for the structure sheaf of the
differential space $(U, \scF^U)$)  with $v_W \in \cDer( (\scF^U)_U(W)
) = \cDer(\scF_M(W))$ for each $W\in \Open(U)$.    We set
\[
\rho^U_V (v) = v_V.
\]
It is not hard to see that the presheaf of Lie algebras $U\mapsto
\cDer (\scF_M(U))$ with the restriction maps $\rho^U_V$ defined above
is (isomorphic to) the sheaf $\cin\Dert(\scF_M)$ of $\cin$-derivations
of the structure sheaf $\scF_M$  of the local $\cin$-ringed space $(M,
\scF)$ (cf.\ Definition~\ref{def:3.2level2}).
\end{remark}  

To prove Theorem~\ref{thm:a1} we need to recall that differential
spaces have bump functions.  In fact existence of bump functions is
equivalent to the condition that the topology on a differential space
$(M,\scF)$ is the smallest topology on $M$ making all the functions
in $\scF$ continuous.  I learned this fact from Yael Karshon.
Definition~\ref{def:a.31} and Lemma~\ref{lem:bump}  below are from \cite{KL}.

\begin{definition} \label{def:a.30} Let $(M,\scT)$ be a topological space,
  $C\subset M$ a closed set and $x\in M\smallsetminus C$ a point.  A
  {\sf bump function} (relative to $C$ and $x$) is a continuous
  function $\rho:M\to [0,1]$ so that $\supp \rho\cap C = \varnothing $
  and $\rho$ is identically 1 on a neighborhood of $x$.
\end{definition}

\begin{definition} \label{def:a.31} Let $(M,\scT)$ be a topological
  space and $\scF \subseteq C^0(M,\R)$ a collection of continuous
  real-valued functions on $M$. The topology $\scT$ on $M$ is {\sf
    $\scF$-regular} iff for any closed subset $C$ of $M$ and any point
  $x\in M\smallsetminus C$ there is a bump function $\rho\in \scF$
  with $\supp \rho \subset M\smallsetminus C$ and $\rho$ identically 1
  on a neighborhood of $x$.
\end{definition}  

\begin{lemma}\label{lem:bump}
Let $(M,\scT)$ be a topological space and $\scF \subset
  C^0(M,\R)$ a $\cin$-subring. Then $\scT$ is the smallest topology making all
  the functions in $\scF$ continuous if and only if the topology $\scT$ is $\scF$-regular.
\end{lemma}  
  
\begin{proof}
Let $\scT_\scF$ denote the smallest topology making all the functions
in $\scF$ continuous.  The set 
\[
  \scS:= \{ f^{-1}(I) \mid  f \in \scF,  \text{ $I$ is an open
    interval } \}
\]  
is a sub-basis for $\scT_\scF$.  Since all the functions in $\scF$ are
continuous with respect to $\scT$, $\scT_\scF\subseteq \scT$.
Therefore it is enough to argue that $\scT\subseteq \scT_\scF $ if and
only if $\scT$ is $\scF$-regular.\\[4pt]
($\Rightarrow$)\quad Suppose $\scT \subseteq \scT_\scF$.  Let
$C\subset M$ be $\scT$-closed and $x$ a point in $M$ which is not in
$C$.  Then $M \setminus C$ is $\scT$-open.  Since $\scT \subseteq
\scT_\scF$ by assumption, $M \setminus C$ is in $\scT_\scF$.   
Then there exist functions $h_1,\ldots,h_k \in \scF$
and open intervals $I_1,\ldots,I_k$ such that 
$x \in \cap_{i=1}^k h_i^{-1}(I_i) \subset M \setminus C$.  There is a
$\cin$ function $\rho: \R^k \to [0,1]$ 
with $\supp \rho \subset I_1 \times \ldots \times I_k$ and the
property that 
$\rho = 1$ on a neighborhood of $(h_1(x),\ldots,h_k(x))$ in $\R^k$.
Then $\tau: = \rho \circ (h_1,\ldots,h_k)$ is in $ \scF$, since $\scF$ is a
$\cin$-subring of $C^0(M)$.  The function $\tau$ is a desired bump
function.\\[4pt]
($\Leftarrow$)\quad Suppose the topology $\scT$ is
$\scF$-regular. Let $U\in \scT$ be an open set.   Then $C= M\setminus
U$ is closed.  Since $\scT$ is $\scF$-regular, for any $x\in U$ there
is a bump function $\rho_x \in \scF$ with $\supp \rho_x \subset U$ and
$\rho_x (x) = 1$.   Then $\rho_x\inv ((0,\infty)) \subset U$
and $\rho_x\inv ((0,\infty)) \in \scT_\scF$.   It follows that
\[
U = \bigcup_{x\in U} \rho_x\inv ((0,\infty)) \in \scT_\scF.
\]  
Since $U$ is an arbitrary element of $\scT$, $\scT \subseteq \scT_\scF$.
\end{proof}

\begin{proof}[Proof of Theorem~\ref{thm:a1}]
  It is easy to see that $\Gamma$ is $\R$-linear and preserves
  brackets; this fact follows directly from the definition of the Lie
  algebra structure on $\cDer(\scF_M)$.  Therefore it is enough to
  show that $\Gamma$ is a bijection.

We start
by proving injectivity.  Suppose $v=\{v_U\}_{U\in \Open(M)} \in
\cDer(\scF_M)$ and $v_M =0$.  We argue that $v_U=0$ for all $U\in
\Open(M)$. By definition of $v$, for any pair of
open sets $V,U\subset M$ with $V\subset U$, for any $h\in \scF_M(U)$
\begin{equation} \label{eq:A.4.1}
   v_V(h|_V) = \left( v_U(h)\right)|_V.
\end{equation}
Hence, since $v_M =0$, for any $f\in \scF = \scF_M(M)$
\[
0 = v_W(f|_W)
\]
 for any $W\in \Open(M)$.
Now given $U\in \Open (M)$ and $\varphi\in \scF_M(U)$ there is an open
cover $\{U_\alpha\}_{\alpha \in A}$ of $U$ and $\{f_\alpha
\}_{\alpha\in A} \subset \scF$  %
so
that $\varphi|_{U_\alpha} = f_\alpha |_{U_\alpha}$ for all $\alpha$
(cf.\ Remark~\ref{rmrk:2.57}).
Then
\[
\left(v_U(\varphi)\right)|_{U_\alpha} \stackrel{\eqref{eq:A.4.1}}{=}
  v_{U_\alpha} (\varphi|_{U_\alpha}) =  v_{U_\alpha}
  (f_\alpha|_{U_\alpha})  \stackrel{\eqref{eq:A.4.1}}{=} v_M(f_\alpha)
    |_{U_\alpha} =0.
\]  
Therefore $v_U(\varphi) =0$ for all $\varphi$. Hence $v_U= 0$ for
all $U\in \Open(M)$.  Thus $v=0$, and $\Gamma$ is injective.

We now argue that $\Gamma$ is surjective: given $w\in \cDer(\scF)$
there is $v= \{v_U\}_{U\in \Open(M)}\in \cDer (\scF_M)$ so that $v_M =w$.
We start by defining the components $v_U$ of $v$ on restrictions of
functions in $\scF$: given $U\in \Open(M)$ and $f\in \scF$ define
$v_U(f|_U)$ by
\[
v_U (f|_U) := w(f)|_U.
\]  
We need to show that $v_U (f|_U) $ is well-defined.  This amounts to
checking that the derivation $w$ is ``local'': for any $f\in \scF$, if
$f|_U=0$ then $w(f)|_U = 0$ as well.  This locality property of
derivations on differential spaces is well-known.  We give a proof to
spare the reader the need to consult other sources.  By
Lemma~\ref{lem:bump} given $x\in U$ there is a bump function $\rho \in
\scF$ with $\rho$ identically 1 near $x$ and support of $\rho$
contained in $U$.   Then $\rho f$ is identically zero, hence $0= w(\rho
f) = f w(\rho) + \rho w(f)$.  Evaluating both sides of the equation on
$x$ and using the fact that $f(x) =0$ while $\rho (x) =1$ we get
\[
0= 1\cdot  w(f) \, (x) + 0\cdot w(\rho)\, (x) = w(f)\,(x).
\]  
Since $x$ is arbitrary this proves that $w(f)|_U =0$, i.e.,
``locality.''

Now given $\varphi\in \scF_M(U)$ there is an open cover
$\{U_\alpha\}_{\alpha \in A}$ of $U$ and $f_\alpha \in \scF$ so that
$\varphi|_{U_\alpha} = f_\alpha |_{U_\alpha}$ for all $\alpha$.
Let $U_{\alpha  \beta}:= U_\alpha \cap U_\beta$.  Then
\[
f_\alpha |_{U_{\alpha  \beta}} = \varphi|_{U_{\alpha  \beta}} =
f_\beta |_{U_{\alpha  \beta}} .
\]  
Locality of $w$ then implies that
\[
w (f_\alpha) |_{U_{\alpha  \beta}}  = w(f_\beta) |_{U_{\alpha  \beta}} 
\]  
for all pairs of indices $\alpha, \beta$.  Therefore it makes sense to
define $v_U(\varphi)$ by setting 
\[
v_U(\varphi)|_{U_{\alpha} }:= w(f_\alpha)|_{U_\alpha}
\]  
for all $\alpha$.  It is not difficult to check that $v_U:\scF_M(U)\to
\scF_M(U)$ so defined is a $\cin$-derivation.   Furthermore by passing
to common refinements one can check that the
definition of $v_U$ does not depend on the choice of a cover
$\{U_\alpha\}_{\alpha \in A}$ or the collection $\{f_\alpha\}_{\alpha
  \in A}$.  In particular if $U= M$ we may take the cover to be the
singleton $\{M\}$ and then $v_M = w$.
Finally one checks that for all pairs of open subsets $V,
U$ of $M$ with $V\subset U$ and all $\varphi \in \scF_M(U)$
\[
v_V(\varphi|_V) = v_U(\varphi)|_V
\]  
for all $\varphi \in \scF_M(U)$.  Therefore the collection
$\{v_U\}_{U\in \Open(M)}$ is the desired  element of $\cDer( \scF_M)$
and the global sections map $\Gamma$ is surjective.
\end{proof}

\end{document}